\documentclass[twoside,10pt]{article}

\usepackage[T1]{fontenc}
\usepackage{times}
\usepackage{amsmath,amsfonts,amssymb,amsthm,euscript,pifont}
\usepackage{mathrsfs}
\usepackage{authblk}
\usepackage[colorlinks=true]{hyperref}
\hypersetup{
    linkcolor = {blue}
    }
\usepackage{fullpage}
\usepackage{xspace}

\newtheorem{theorem}{Theorem}[section]
\newtheorem{lemma}[theorem]{Lemma}
\newtheorem{proposition}[theorem]{Proposition}
\newtheorem{remark}[theorem]{Remark}

\newcommand{\Dt}{{\Delta t}}
\newcommand{\FF}{{\mathcal F}}
\newcommand{\LL}{{\mathcal L}}
\newcommand{\OO}{{\mathcal O}}
\newcommand{\Oe}{{\OO_{\exp}}}

\newcommand{\ZZ}{{\mathcal Z}}

\usepackage{dsfont}
\def\ind{\mathds{1}}
\newcommand{\sgn}{\displaystyle {\rm sgn}}

\newcommand{\Ee}{\displaystyle {\mathbb E}}
\newcommand{\Pp}{\displaystyle {\mathbb P}}
\newcommand{\Qq}{\displaystyle {\mathbb Q}}
\newcommand{\er}{\displaystyle {\mathbb R}}
\newcommand{\nat}{{\mathbb N}}

\newcommand{\oX}{{\overline {X}}}
\newcommand{\oZ}{{\overline {Z}}}
\newcommand{\tD}{{\tau }}
\newcommand{\Verr}{{\mathcal E}}
\usepackage{mathrsfs}

\newcommand{\Oomega}{{\omega  }}

\begin{document}
\title{Weak convergence analysis of the symmetrized Euler scheme  for one dimensional SDEs with diffusion  coefficient $|x|^\alpha $, $\alpha \in [\tfrac{1}{2},1)$\footnote{A previous version of this paper circulated with the title:  {\it An efficient discretisation scheme for one dimensional SDEs with a diffusion coefficient function of the form $|x|^\alpha $, $\alpha \in [\tfrac{1}{2},1)$ , Inria research report  No-5396.}}}

\author{Mireille Bossy\thanks{email: mireille.bossy@inria.fr}~}
\author{Awa Diop}
\affil{TOSCA Laboratory, INRIA Sophia Antipolis -- M\'editerran\'ee, France}
\date{November, 2010}

\maketitle

\begin{abstract}
In this paper, we are interested in the time discrete approximation of
$\Ee f(X_T)$  when $X$ is the solution of a stochastic differential equation 
with a diffusion coefficient function of the form $|x|^\alpha $.
We propose a symmetrized version of the Euler scheme, applied to  $X$. The symmetrized version  is very easy to simulate on a computer. 
For smooth functions $f$,  we prove the Feynman Kac
representation formula $u(t,x) = \Ee_{t,x} f(X_T)$, for $u$ solving the
associated Kolmogorov PDE and we obtain the upper-bounds on the spatial
derivatives of $u$  up to the order four.  Then we show that the weak error
of our symmetrized scheme is of order one, as for the classical Euler scheme. 
\end{abstract}

\section{Introduction}

We consider $(X_t,t\geq 0)$, the $\er$-valued process solution to
the following one-dimensional It\^o stochastic 
differential equation 
\begin{eqnarray}\label{modele}
X_t = x_0 + \int _0^t b(X_s)ds + \sigma \int _0^t |X_s| ^\alpha  dW_s, 
\end{eqnarray}	
where $x_0$ and $\sigma $ are given constants, $\sigma  >0$ and 
$(W_t,t\geq 0)$ is a one-dimensional Brownian motion defined on a
given probability space $(\Omega , \FF, \Pp)$. We denote by $(\FF_t,t\geq
0)$ the Brownian filtration.  
To ensure the existence of such process, we state the following 
hypotheses:
\begin{description}
\item[(H0)] {\it $\alpha  \in [1/2,1)$. } 
\item[(H1)] {\it The drift function $b$ is such that $b(0) > 0$ and
satisfies the Lipschitz condition}
\begin{eqnarray*} 
\left| b(x) - b(y)\right| \leq K |x - y|,~\forall  ~(x,y)\in \er^2. 
\end{eqnarray*} 
\end{description}
Under hypotheses (H0) and (H1), strong existence and uniqueness holds for
equation \eqref{modele}. Moreover, when $x_0 \geq 0$ and $b(0) >0$, the process
$(X_t,t\geq 0)$ is valued in $[0,+\infty )$ (see e.g. \cite{karatzas-shreve-88}).
Then $(X)$ is the unique strong solution to
\begin{align}\label{modele-positif}
X_t = x_0 + \int _0^t b(X_s)ds + \sigma \int _0^t X_s ^\alpha  dW_s. 
\end{align}

Simulation schemes for Equation \eqref{modele} are motivated by some 
applications in  Finance: in \cite{cox-ingersoll-al-85}, Cox, Ingersoll and
Ross (CIR) proposed to model the dynamics of the short term interest rate
as the solution of  \eqref{modele} with $\alpha =1/2$ and $b(x) = a -bx$.
Still to model the short term interest rate, Hull and White
\cite{hull-white-90} proposed the following mean-reverting
diffusion process 
\begin{align*}
dr_t = (a(t) - b(t) r_t) dt + \sigma (t) r_t^\alpha  dW_t
\end{align*}
with $0\leq \alpha  \leq 1$. More recently, the stochastic--$\alpha \beta \rho $ model
or $SABR$--model have been
proposed  as a stochastic correlated volatility model for the asset price 
(see \cite{hagan-al-02}):
\begin{align*}
&dX_t = \sigma _t X_t^\beta  dW^1_t \\
&d\sigma _t = \alpha \sigma  _t dB_t
\end{align*}
where $B_t = \rho   W^1_t + \sqrt {(1-\rho^2)} W^2_t$, $\rho \in[-1,1]$ and $(W^1,W^2)$ is a
2d--Brownian motion.

CIR-like models arise also in fluid mechanics: in the stochastic
Lagrangian modeling of turbulent flow, characteristic quantities like the
instantaneous turbulent 
frequency $(\omega  _t)$ are  modeled by (see \cite{dreeben-pope-97})
\begin{align*}
d\omega  _t = - C_3 \langle \Oomega_t \rangle\left (\omega  _t - \langle
\Oomega_t \rangle\right) dt - S(  \langle \Oomega_t 
\rangle) \omega _t dt \\
+ \sqrt {C_4 \langle
\Oomega_t \rangle^2\omega _t }dW_t
\end{align*}
where the ensemble average $\langle \Oomega_t \rangle$ denotes here
the  conditional expectation
with respect to the position of the underlying  portion of fluid and
$S(\omega )$ is a given function.  

In the examples above, the solution processes are all positive. In the 
practice, this could be an important feature of the model that
simulation procedures  have to preserve. 
 By using the classical Euler scheme, one cannot define a positive
approximation process.
Similar situations  occur when one consider discretisation scheme of a
reflected stochastic differential equation. To maintain the approximation
process  in a given domain,  an efficient strategy consists in 
symmetrizing  the value obtained by the Euler scheme with respect to the 
boundary of the domain (see e.g. \cite{bossy-gobet-al-04}).
Here, our preoccupation is quite similar. We want to maintain the
positive value of the approximation. In addition, we have to deal with a
just locally Lipschitz diffusion coefficient. 

In \cite{deelstra-delbaen-98}, Deelstra and Delbaen prove the strong
convergence of the Euler scheme apply to $dX_t = \kappa (\gamma  -X_t) dt + g(X_t) 
dW_t$ where $g:\er\rightarrow \er^+$ vanishes at zero and satisfies
the H\"older condition $|g(x) - g(y) | \leq b\sqrt {|x-y|}$. The Euler scheme 
is applied to the modified equation  $dX_t = \kappa (\gamma  -X_t) dt +
g(X_t\ind_{\{X_t \geq 0\}})dW_t$.  This corresponds to a
projection scheme.   For reflected SDEs, this procedure
convergences weakly with a rate $\frac{1}{2}$ (see
\cite{costantini-pacchiarotti-al-98}). Moreover, the positivity of the
simulated process is not guaranteed.
In the particular case of the CIR processes, Alfonsi \cite{alfonsi-05} proposes  some implicit schemes, which admit    analytical solutions, and derives from them a family of explicit schemes.
He 
analyses their  rate of convergence (in both strong and weak sense) 
and  proves a weak rate of convergence of order 1 and an error expansion in the
power of the time-step for the explicit family.  
Moreover, Alfonsi provides an interesting  numerical comparison
between the Deelstra and Delbaen scheme, his schemes and the present
one discussed in this paper, in the special case of CIR processes. 

In section \ref{Scheme}, we construct our time discretisation scheme for
$(X_t,t\in[0,T])$, based on the symmetrized  Euler scheme and which can
be simulated easily. 
We prove a theoretical rate of convergence of order
one for the weak approximation error. We  analyze separately the cases  
$\alpha =1/2$ and $1/2< \alpha  <1$. The convergence results  are given in the next section in Theorems
\ref{theorem-faible-CIR-gene} and \ref{theorem-faible-HUW-gene}
respectively.  The  sections \ref{Cas1} and \ref{Cas2} are
devoted to the proofs in this two respective situations. 
We denote $(\oX_t,t\in[0,T])$  the
approximation process.   
To study the weak error $\Ee f(X_T) - \Ee f(\oX_T)$, we will use the
 Feynman--Kac representation $\Ee f(X_{T-t}^x) = u(t,x)$ where
$u(t,x)$ solves the associated Kolmogorov PDE. 
The two main ingredients of the rate of convergence analysis consist 
in, first obtaining  the upper-bounds on the spatial derivatives of $u(t,x)$ up
to the order four. To our knowledge, for this kind of Cauchy problem, 
there is no generic result.
The second point consists in studying  the behavior of the
approximation process  at the origin.

Let us emphasis the difference between the situations $\alpha =1/2$ and $1/2< \alpha  <1$. 
The case  $\alpha =1/2$ could seem intuitively easier as the associated
infinitesimal generator has unbounded but smooth coefficients. In
fact, studying the spatial derivative of $u(t,x)$ with probabilistic
tools, we need to impose the condition $b(0) > \sigma  ^2$, 
in order to define the derivative of $X^x_t$ with respect to $x$.
In addition, the analysis of the approximation  process  $(\oX)$ at the
origin shows that the expectation of its local time is in
$\Dt^{b(0)/\sigma  ^2}$.

In the case $1/2< \alpha  <1$,  the derivatives of the
diffusion coefficient of the associated infinitesimal generator are
degenerated functions at point zero.  As we cannot hope to obtain uniform 
upper-bounds in $x$ for the derivatives of $u(t,x)$, we prove that the
approximation process goes to a neighborhood of the origin with an
exponentially small  probability and  we give 
upper bounds for the negative moments of the approximation
process $(\oX)$. 
\section{The symmetrized  Euler scheme for~\eqref{modele}}\label{Scheme}
For $x_0 \geq 0$, let $(X_t,t\geq 0)$ given by~\eqref{modele}
or~\eqref{modele-positif}.
For a fixed time $T>0$, we define a
discretisation scheme $(\oX_{t_k}, k=0,\ldots ,N)$ by 
\begin{align}\label{schema}
\left\{
\begin{array}{l}
\oX_0 = x_0 \geq 0,\\
\oX_{t_{k+1}}=\left| \oX_{t_k} + b(\oX_{t_k})\Dt +
\sigma  \oX_{t_k}^\alpha  (W_{t_{k+1}}-W_{t_k})\right|,
\end{array} \right.
\end{align} 
$k=0,\ldots ,N-1$, where $N$ denotes the number of 
discretisation times $t_k = k \Dt$ and $\Dt>0$ is a constant time step
such that $N \Dt = T$. 

In the sequel we will use the time continuous version
$(\oX_{t},0\leq t\leq T)$
of the discrete time process, which consists in freezing the
coefficients on each interval
$[t_k,t_{k+1})$:
\begin{align}\label{schema-continu}
\oX_t = \left| \oX_{\eta (t)} + (t - \eta (t)) b(\oX_{\eta (t)}) + \sigma 
\oX_{\eta (t)}^\alpha  (W_t -W_{\eta (t)})\right|,
\end{align} 
where $\eta (s) = \sup_{k\in \{1,\ldots ,N\}}\{t_k; t_k \leq s\}$.
The process $(\oX_{t},0\leq t\leq T)$ is valued in
$[0,+\infty )$.  By induction
on each subinterval $[t_k,t_{k+1})$, for $k=0$ to $N-1$, by using the
Tanaka's formula, we can easily show that $(\oX_t)$ is a continuous
semi-martingale with a continuous local time $(L^0_t(\oX))$ at point
$0$. Indeed, for any $t\in [0,T]$, if we set  
\begin{align}\label{Z_t}
\oZ_t = \oX_{\eta (t)} + b(\oX_{\eta (t)})(t-\eta (t))
+ \sigma \oX_{\eta (t)}^\alpha  (W_t-W_{\eta (t)}), 
\end{align}
then $\oX_t=|Z_t|$ and 
\begin{align}\label{schema-continu-bis}
\oX_t =  x_0 + \int _0^t \sgn(\oZ_s) b(\oX_{\eta (s)})ds
+ \sigma  \int _0^t \sgn(\oZ_s)\oX_{\eta (s)}^\alpha  dW_s + \frac{1}{2}L^0_t(\oX), 
\end{align} 
where $\sgn(x) := 1 - 2\,\ind_{(x\leq 0)}$. 

The following lemma ensures the existence of the positive moments of
$(X_t)$,  starting at $x_0$ at time 0, and of $(\oX_t)$, its associated discrete time process:
\begin{lemma}\label{majo-moments-oX}
Assume $\rm (H0)$ and $\rm (H1)$. For any $x_0 \geq 0$, for any $p\geq 1$,
there exists a positive constant $C$, depending on $p$, 
but also on the parameters  $b(0)$, $K$, $\sigma $, $\alpha $ and $T$, such that
\begin{align}\label{moments-oX}
 \Ee \left( \sup_{t\in [0,T]} X_t^{2p}\right) +  \Ee \left(
\sup_{t\in [0,T]} \oX_t^{2p}\right) \leq  C(1 + x_0^{2p}). 
\end{align}
\end{lemma}
In the following proof, as well as in the rest of the paper, $C$
will denote a constant that can change from line to line. $C$ could
depend on  the parameters of the model, but it is always 
independent of $\Dt$.  

\begin{proof}
We prove~\eqref{moments-oX} for $(\oX_t,0\leq t\leq T)$ only, the
case of $(X_t,0\leq t\leq T)$ could be deduced by similar
arguments. 
By the It\^o's formula,  and noting that for any $t\in [0,T]$
$\int _0^t\oX^{2p-1}_s dL^0_s(\oX) = 0$, we have 
\begin{align}\label{for_BDG}
\begin{array}{l}
\oX_t^{2p} =  x_0^{2p} + \displaystyle  2p \int _0^t \oX_s^{2p-1}\sgn(\oZ_s)
b(\oX_{\eta (s)}) ds  \\
+ 2p\sigma  \int _0^t \oX^{2p-1}_s \sgn(\oZ_s) \oX_{\eta (s)}^\alpha 
dW_s + \sigma  ^2   p(2p-1)\displaystyle  \int _0^t \oX^{2p-2}_s
\oX^{2\alpha  }_{\eta (s)} ds.
\end{array}	
\end{align} 
To prove~\eqref{moments-oX}, let's start by showing that
\begin{align}\label{moments-oX-bis}
\sup_{t\in [0,T]} \Ee \left(
\oX_t^{2p}\right) \leq  C(1 + x_0^{2p}). 
\end{align}
\eqref{moments-oX} will follow  from~\eqref{moments-oX-bis},
\eqref{for_BDG} and the Burkholder-Davis-Gundy Inequality. 
Let $\tau _n$ be the stopping time defined by
$\tau _n = \inf\{ 0<s<T; \oX_s \geq n\}$, with 
$\inf\{\emptyset  \} = 0$. Then, 
\begin{align*}
\Ee \oX_{t\wedge  \tau _n}^{2p} \leq  x_0^{2p}
+ 2p \Ee\left(\int _0^{t\wedge  \tau _n} \oX_s^{2p-1}b(\oX_{\eta (s)})  ds \right)
+ \sigma  ^2   p(2p-1)\Ee\left(\int _0^{t\wedge  \tau _n}
\oX^{2p-2}_s \oX^{2\alpha }_{\eta (s)} ds\right). 
\end{align*}
By using $\rm (H0)$, $\rm (H1)$ and the Young Inequality, we get
\begin{align*}
\Ee \oX_{t\wedge  \tau _n}^{2p} \leq & x_0^{2p} + T b(0)^{2p}  
+ (2p-1)\Ee\left(\int _0^{t\wedge  \tau _n} \oX^{2p}_s ds\right) \\ 
& + 2pK  \Ee\left( \int _0^{t\wedge  \tau _n}\oX_s^{2p-1}\oX_{\eta (s)}  ds  \right)
+ \sigma  ^2   p(2p-1)\Ee\left(\int _0^{t\wedge  \tau _n} 
\oX^{2p-2}_s \oX^{2\alpha  }_{\eta (s)} ds\right). 
\end{align*}
Replacing $\oX_s$ by~\eqref{schema-continu} in the
integrals above,  by using another time $\rm (H1)$ and
the Young Inequality, we easily obtain that for any $t\in[0,T]$, 
\begin{align*}
\Ee \oX_{\eta (t)\wedge  \tau _n}^{2p}
\leq  x_0^{2p} + C \left( 1 + \int _0^{\eta (t)}
\Ee\left( \oX^{2p}_{\eta (s)\wedge  \tau _n}\right) ds\right), 
\end{align*}
where $C>0$ depends on $p$, $b(0)$, $K$, $\sigma $, $\alpha $ and $T$. 
A discrete version of the Gronwall Lemma allows us to conclude that
\begin{align*}
\sup_{k = 0,\ldots  ,N} \Ee \left(\oX_{t_k \wedge  \tau _n}^{2p} \right) \leq
C (1 + x_0^{2p}), 
\end{align*}
for another constant  $C$, which does not depend on $n$. Taking the limit
$n \rightarrow +\infty $, we get that 
$\sup_{k = 0,\ldots  ,N} \Ee (\oX_{t_k}^{2p} )\leq C(1 + x_0^{2p})$, from
which we easily deduce~\eqref{moments-oX-bis} using~\eqref{schema-continu}.  
\end{proof} 

\subsection{Main results}

In addition of hypotheses {\rm (H0)} and {\rm (H1)}, we will analyze the
convergence rate of~\eqref{schema} under the following hypothesis: 
\begin{description}
\item[(H2)] {\it The drift function $b(x)$ is a $C^4$ function, with bounded
derivatives up to the order 4.}
\end{description}
\subsubsection{Convergence rate when $\alpha  =1/2$}

Under {\rm (H1)}, $(X_t,0\leq t\leq T)$ satisfies 
\begin{align}\label{CIR_gene}
X_t = x_0  + \int_0^t b(X_s) ds +\sigma  \int _0^t \sqrt {X_s} dW_s,\;\;0\leq t\leq T.
\end{align}
When $b(x)$ is of the form $a - \beta   x$,  with $a>0$, $(X_t)$ is the
classical CIR process used in Financial mathematics to model the
short interest rate.  When $b(x) = a >0$, $(X_t)$ is the square of
a Bessel process.  Here we consider a generic drift function $b(x)$,
with the following restriction : 
\begin{description}
\item[(H3)] $b(0) >  \sigma  ^2$.
\end{description}

\begin{remark}\label{H3}
When $x_0>0$ and $b(0) \geq \sigma  ^2/2$, by using the Feller's test, one
can show that 
$\Pp(\tau  _0 = \infty  ) =1$ where $\tau _0 = \inf \{t\geq 0; X_t = 0\}$. 
We need the stronger Hypothesis {\rm (H3)} to prove that the
derivative (in the sense of the quadratic mean) of  
$X^x_t$ with respect to $x$ is well defined (see
Proposition~\ref{justification-cir} and its proof
in Appendix~\ref{proofs-justification}). In particular, we need to use  the
Lemma~\ref{moment-inverse-cir-gene} which controls the
inverse moments and the exponential inverse moment of the
CIR--like process $(X_t)$, for some values of the 
parameter $\nu  = \frac{2 b(0)}{\sigma  ^2} -1>1$.
\end{remark}

Section \ref{Cas1} is devoted to the proof of the following 
\begin{theorem}\label{theorem-faible-CIR-gene}
Let $f$ be a $\er$-valued 
$C^4$ bounded function,  with bounded spatial derivatives up to the
order $4$. Let $\alpha  =\frac{1}{2}$ and $x_0 >0$. 
Assume  {\rm (H1), (H2)} and {\rm (H3)}. Choose  $\Dt$ 
sufficiently  small in~\eqref{schema}, i.e. $\Dt \leq 1/
(2K)\wedge  x_0$. Then there exists a positive 
constant $C$ depending on $f$, $b$, $T$ and $x_0$ such that
\begin{align*}
\left| \Ee f(X_T) -\Ee f(\oX_T) \right| 
\leq C \left(\Dt + \left(\frac{\Dt}{x_0}\right)^{\frac{b(0)}{\sigma  ^2}}\right). 
\end{align*}
\end{theorem}
Under {\rm (H3)}, the  global theoretical rate of convergence is of
order one. When $b(0) < \sigma  ^2$, numerical tests for the CIR process show
that the rate of convergence becomes under-linear (see \cite{diop-thesis} and the
comparison of numerical schemes for the CIR process performed by
Alfonsi in \cite{alfonsi-05}).

\subsubsection{Convergence rate when $1/2< \alpha <1$}

Under {\rm (H1)}, $(X_t,0\leq t\leq T)$ satisfies 
\begin{align}\label{HUW_gene}
X_t = x_0  + \int_0^t b(X_s) ds +\sigma  \int _0^t X_s^{\alpha } dW_s,\;\;0\leq t\leq T. 
\end{align}
We restrict ourselves to the case 
\begin{description}
\item[(H3')] $x_0 > \frac{b(0)}{\sqrt {2}} \Dt$.
\end{description}

\begin{remark}
When $x_0>0$, the Feller's test on process $(X_t)$ shows that it is enough to
suppose $b(0) >0$, as in {\rm (H1)}, to ensure that $\Pp(\tau  _0 = \infty ) =1$, for
$\tau _0 = \inf\{t\geq 0; X_t = 0\}$.
\end{remark} 

In Section \ref{Cas2}, we prove the following 
\begin{theorem}\label{theorem-faible-HUW-gene}
Let $f$ be a $\er$-valued bounded 
$C^4$ function,  with bounded spatial derivatives up to the order
$4$. Let $\frac{1}{2}<\alpha <1$. 
Assume  {\rm (H1), (H2)} and {\rm (H3')}.   Choose  $\Dt$ 
sufficiently  small in~\eqref{schema}, i.e.  $\Dt \leq 1/ (4K)$. Then there
exists a positive constant $C$ depending on $f$,  $\alpha $, $\sigma  $, $b$, $T$
and $x_0$ such that 
\begin{align*}	
|\Ee f(X_T)- \Ee f(\oX_T)|
\leq C \left(1+\frac{1}{x_0^{q(\alpha )}}\right)\Delta  t,
\end{align*} 
where $q(\alpha )$ is a positive constant depending only on $\alpha $.
\end{theorem}
\section{The case of processes with $\alpha  =1/2$}\label{Cas1}

\subsection{Preliminary results}
In this section $(X_t)$ denotes the solution of~\eqref{CIR_gene}
starting at the deterministic point $x_0$ at time 0. When we need to  
vary the deterministic initial position, we mention it explicitly by  using the 
notation $(X_t^{x})$ corresponding to the 
unique strong solution of the equation 
\begin{align}\label{CIR_gene_flot}
X_t^{x} =  x + \int _{0}^t b(X_s^{x})ds
+\sigma  \int _{0}^t\sqrt {X_s^{x}}dW_s. 
\end{align} 
\subsubsection{On the exact process}
We have the following
\begin{lemma}\label{moment-inverse-cir-gene}
Let us assume {\rm (H1)} and {\rm (H3)}.  We  set $\nu  =\frac{2b(0)}{\sigma  ^2}-1
>1$. 
For any $p$ such that $1< p < \nu $, for any $t\in [0,T]$ and any $x>0$,  
\begin{align*} 
\Ee\left(X^x_t\right)^{-1} \leq C(T) x^{-1}\mbox{ and }
\Ee\left(X^x_t\right)^{-p} \leq C(T)
t^{-p}\mbox{ or }\Ee\left(X^x_t\right)^{-p} \leq C(T,p) x^{-p}.
\end{align*} 
Moreover for all $\mu  \leq \frac{\nu^2 \sigma  ^2}{8},$
\begin{align}\label{aux_0}	
\Ee\exp\left(\mu  \int _0^t (X^x_s)^{-1} ds\right)\leq
C(T) \left(1 + {x}^{-\nu /2}\right),
\end{align} 
where the positive constant $C(T)$ is a non-decreasing function of
$T$ and does not depend on $x$.
\end{lemma}
\begin{proof}
As $b(x)\geq b(0)-Kx$, The Comparison Theorem gives that, a.s. 
$X^x_t \geq Y^{x}_t$, for all
$t\geq 0 $, where $(Y^x_t,t\leq T)$ is the  CIR process solving  
\begin{align}\label{CIR_gene_0} 
Y^{x}_t =  x + \int _0^t \left(b(0)-K Y^{x}_s\right)ds
+ \sigma \int _{0}^t  \sqrt {Y^x_s} dW_s. 
\end{align} 
In particular, $\Ee \exp (\mu  \int _{0}^t (X_s^{x})^{-1} ds)
\leq \Ee \exp (\mu  \int _{0}^t (Y_s^{x})^{-1} ds)$.
As $b(0)> \sigma  ^2$ by {\rm (H3)}, one can derive~\eqref{aux_0} from the
Lemma~\ref{novikov-cir}. Similarly, for the upper bounds on the inverse moments of
$X^x_t$, we apply the Lemma~\ref{moment-inverse-cir}. 
\end{proof} 

\subsubsection{On the associated Kolmogorov PDE}
\begin{proposition}\label{proposition_edp_alpha=1/2}
Let $\alpha  =1/2$. Let $f$ be a $\er$-valued 
$C^4$ bounded function,  with bounded spatial derivatives up to the
order $4$. 
We consider the $\er$-valued function  defined on 
$[0,\,T]\times [0,+\infty )$ by $u(t,x)=\Ee f(X^x_{T-t})$. 
Under {\rm (H1), (H2)} and {\rm (H3)}, $u$ is in
$C^{1,4}([0,T]\times  [0,+\infty ))$. That is, $u$ has a first derivative in the time
variable and derivatives up to order 4 
in the space variable. Moreover, there
exists a positive constant  
$C$ depending on $f$, $b$ and $T$ such that, for all $x\in [0,+\infty )$, 
\begin{align}	
\displaystyle  \sup_{t\in[0,T]}\displaystyle  \left|\frac{\partial  u}{\partial t}(t,x)\right|\leq & C(1 + x),\label{borne_u_alpha=1/2}\\
\|u \|_{L^\infty  \left([0,T]\times [0,+\infty )\right)} +
\sum _{k=1}^4 \left\|\frac{\partial ^k u}{\partial  x^k}\right\|_{L^\infty  \left([0,T]\times 
[0,+\infty )\right)}\leq & C \label{borne_deriv_u_alpha=1/2}
\end{align} 
and $u(t,x)$ satisfies 
\begin{align}\label{edp_u_alpha=1/2}
\left\{\begin{array}{l}
\displaystyle  \frac{\partial  u}{\partial  t}(t,x) + b(x)\frac{\partial  u}{\partial  x}(t,x) +\frac{\sigma  ^2}{2}x
\frac{\partial ^2u}{\partial  x^2}(t,x)=0,~(t,x)\in [0,T)\times [0,+\infty  ), \\
\displaystyle  u(T,x)=f(x),\; x\in [0,+\infty  ).
\end{array} \right. 
\end{align} 
\end{proposition}
In all what follows, we will denote $\|~\|_{\infty }$ the norm on $L^\infty $
spaces. 
Before to prove the Proposition~\ref{proposition_edp_alpha=1/2}, we 
introduce some notations and give preliminary results:
for any $\lambda \geq 0$ and any $x > 0$,  we denote by 
$(X_t^{x}(\lambda ),0\leq t\leq T)$,  the  $[0,+\infty )$-valued process solving 
\begin{align}\label{Xlambda}
X_t^{x}(\lambda ) =  x + \lambda \sigma ^2t + \int _{0}^t b(X_s^{x}(\lambda ))ds
+\sigma  \int _{0}^t\sqrt {X_s^{x}(\lambda )}dW_s. 
\end{align} 
Equation~\eqref{Xlambda} has a non-exploding unique strong
solution. Moreover, for all $t\geq 0$, $X_t^{x}(\lambda ) \geq  X_t^{x}$. 
The coefficients are locally
Lipschitz on $(0,+\infty )$, with locally Lipschitz first order
derivatives. Then $X_t^{x}(\lambda )$ is continuously differentiable (see
e.g. Theorem V.39 in \cite{protter-90}), and if we denote
$J^x_t(\lambda ) = \frac{dX_t^{x}}{dx}(\lambda )$, 
the process $(J_t^{x}(\lambda ),0\leq t\leq T)$ satisfies the linear
equation
\begin{align}\label{derivee-flot-eds}
J^x_t(\lambda ) = 1+ \int _{0}^t  J^x_s(\lambda ) b'(X_s^{x}(\lambda ))ds
+ \int _{0}^t J^x_s(\lambda ) \frac{\sigma  dW_s}{2\sqrt {X_s^{x}(\lambda )}}. 
\end{align} 
By Lemma~\ref{moment-inverse-cir-gene}, the process
$(\int _0^t\frac{dW_s}{\sqrt {X^x_s(\lambda  )}},0\leq t\leq T)$ is a locally square integrable
martingale. Then, for all $\lambda  \geq 0$,  $J^x_t(\lambda )$
is given by  (see e.g. Theorem V.51 in \cite{protter-90}),
\begin{align}\label{derivee-flot}
J^x_t(\lambda ) =  \exp\left(
\int _{0}^t b'(X_s^{x}(\lambda ))ds +
\frac{\sigma }{2}\int _{0}^t\frac{dW_s}{\sqrt {X_s^{x}(\lambda )}}
-\frac{\sigma ^2}{8}\int _{0}^t\frac{ds}{X_s^{x}(\lambda )}\right). 
\end{align} 
\begin{lemma}\label{Martingale-expo-lambda}
Assume {\rm (H3)}. The process $(M^x_t(\lambda ),0\leq t\leq T)$ defined by 
\begin{align*}
M^x_t(\lambda ) = \exp\left(
\frac{\sigma }{2}\int _{0}^t\frac{dW_s}{\sqrt {X_s^{x}(\lambda)}}
-\frac{\sigma  ^2}{8}\int _0^t\frac{ds}{X_s^{x}(\lambda)}\right)
\end{align*} 
is a $\Pp$-martingale. Moreover,
$\sup_{t\in [0,T]} \Ee \left(J^x_t(\lambda )\right)\leq C$
where the positive constant $C$ does not depend on $x$ . 
\end{lemma}
\begin{proof}
By Lemma~\ref{moment-inverse-cir-gene}, $(M^x_t(\lambda ),0\leq t\leq T)$
satisfies the Novikov criterion. Under {\rm (H2)}, $b'$ is
a bounded function and  
\begin{align*}
\Ee\left[J^x_t(\lambda  )\right] = 
\Ee\left[\exp\left(\int _0^t b'(X^x_s(\lambda ))ds\right) M^x_t(\lambda  )\right]
\leq \exp(\|b'\|_\infty  T). 
\end{align*}
\end{proof}
Let $(\ZZ_t^{(\lambda  ,\lambda +\frac{1}{2})},0\leq t\leq T)$ defined by
\begin{align}\label{def_Z}
\begin{array}{ll}
\ZZ_t^{(\lambda ,\lambda +\frac{1}{2})}
=& \exp\left(-\frac{\sigma  }{2} \int _0^t \frac{1}{\sqrt {X^x_s(\lambda  )}}
\left( \frac{dX^x_s(\lambda  )}{\sigma  \sqrt {X^x_s(\lambda  )}} \right.\right.\\
&\hspace{1.7cm}\left.\left. -
\frac{b(X^x_s(\lambda  )) + (\lambda +\frac{1}{2}) \sigma ^2}{\sigma  \sqrt {X^x_s(\lambda  )}} ds \right)
 \frac{\sigma ^2}{8} \int _0^t \frac{ds}{X^x_s(\lambda  )}\right).
\end{array}	
\end{align} 
By the Girsanov Theorem, under the probability $\Qq^{\lambda +\frac{1}{2}}$
such that
$\frac{ d\Qq^{\lambda +\frac{1}{2}}}{ d\Pp}\bigg|_{\FF_t}
=\frac{1}{Z_t^{(\lambda ,\lambda +\frac{1}{2})}}$, the process 
$(B_t^{\lambda +\frac{1}{2}} = \int _0^t \frac{dX^x_s(\lambda  )}{\sigma  \sqrt {X^x_s(\lambda  )}} -
\frac{b(X^x_s(\lambda  )) + (\lambda +\frac{1}{2}) \sigma ^2}{\sigma  \sqrt {X^x_s(\lambda  )}} ds,t\in[0,T])$ is a
Brownian motion on $(\Omega ,\FF_T,\Qq^{\lambda +\frac{1}{2}})$.
Indeed, we have that
\begin{align*} 
X_t^{x}(\lambda ) =  x + (\lambda +\frac{1}{2})\sigma ^2t + \int _{0}^t b(X_s^{x}(\lambda ))ds
+\sigma  \int _{0}^t\sqrt {X_s^{x}(\lambda )}dB^{\lambda +\frac{1}{2}}_s.
\end{align*} 
Hence, $\LL^{\Qq^{\lambda +\frac{1}{2}}}(X^x_.(\lambda  )) = \LL^{\Pp}(X^x_.(\lambda 
+\frac{1}{2}))$  
and from the Lemma~\ref{Martingale-expo-lambda},
$\ZZ_t^{(\lambda ,\lambda +\frac{1}{2})} = \exp(-\frac{\sigma  }{2} \int _0^t
\frac{dB^{\lambda +\frac{1}{2}}_s}{\sqrt {X_s(\lambda  )}} - \frac{\sigma ^2}{8} \int _0^t
\frac{ds}{X_s(\lambda  )})$ is a $\Qq^{\lambda +\frac{1}{2}}$--martingale. 
The following proposition allows us to compute the derivatives of
$u(t,x)$.
\begin{proposition}\label{justification-cir}
Assume {\rm (H1)}, {\rm (H2)} and {\rm (H3)}. 
Let $g(x)$, $h(x)$ and $k(x)$ be some  bounded $C^1$ functions, with
bounded first derivatives. 
For any $\lambda  \geq 0$, let $v(t,x)$  be the $\er$-valued function defined, on
$[0,T]\times \er_+^*$, by 
\begin{align*} 
v(t,x) = \Ee \left[g(X_t^x(\lambda  )) \exp(\int _0^t k(X_s^x(\lambda  )) ds)\right] 
+ \int _0^t \Ee\left[h(X_s^x(\lambda  )) \exp(\int _0^s k(X_\theta ^x (\lambda  )) d\theta  )\right] ds.
\end{align*} 
Then $v(t,x)$ is of class $C^1$ with
respect to $x$ and
\begin{align*} 
\frac{\partial  v}{\partial  x}(t,x) = 
\Ee \left[\exp(\int _0^t k(X_s^x(\lambda  )) ds)\left(g'(X_t^x(\lambda  ))J^x_t(\lambda  ) +
g(X_t^x(\lambda  )) \int _0^t k'(X_s^x(\lambda  ))J^x_s(\lambda  ) ds\right)\right] \\
 + \int _0^t \Ee\left[\exp(\int _0^s k(X_\theta ^x(\lambda  ) )
d\theta )\left(h'(X_s^x(\lambda  ))J^x_s(\lambda  ) + h(X_s^x(\lambda  ))\int _0^s k'(X_\theta ^x (\lambda  ))J^x_\theta (\lambda  )
 d\theta \right)\right] ds.
\end{align*}
\end{proposition}
The proof is postponed in the Appendix~\ref{proofs-justification}.

\begin{proof}[Proof of Proposition~\ref{proposition_edp_alpha=1/2}]
First, we note that $u(t,x)=\Ee f(X_{T-t}^x)$ is a continuous function 
in $x$ and bounded  by $\|f\|_\infty  $.  
Let us show that $u$ is in $C^{1,4}([0,T]\times [0,+\infty ))$. 
$f$ being in $C^4(\er)$, by the It\^o's formula, 
\begin{align*}
u(t,x)
&= f(x)+\int _0^{T-t} \Ee \left( b(X_s^{x})f'(X_s^{x})\right) ds
+ \frac{\sigma ^2}{2}\int _0^{T-t}\Ee\left( X_s^{x} f''(X_s^{x})\right) ds \\
&+ \sigma \Ee\left( \int _0^{T-t} \sqrt {X_s^{x}} f'(X_s^{x})dW_s\right). 
\end{align*} 
$f'$ is bounded and $(X^x_t)$ have moments of any order. Then  we
obtain that
\begin{align*} 
\frac{\partial  u}{\partial 
t}(t,x) = - \Ee \left(b(X_{T-t}^{x}) f'(X_{T-t}^{x})
+\frac{\sigma ^2}{2} X_{T-t}^{x} f''(X_{T-t}^{x})\right). 
\end{align*} 
Hence, $\frac{\partial  u}{\partial t}$ is a continuous function on $[0,T]\times [0,+\infty )$
and~\eqref{borne_u_alpha=1/2} follows by 
Lemma~\ref{majo-moments-oX}. 

By Proposition~\ref{justification-cir}, for $x>0$, 
$u(t,x)= \Ee f(X_{T-t}^x)$ is differentiable  and 
\begin{align*} 
\frac{\partial  u}{\partial  x}(t,x) =
\Ee\left( f'(X_{T-t}^x(0))J_{T-t}^x(0) \right). 
\end{align*} 
Hence, by using the Lemma~\ref{Martingale-expo-lambda},
$\left|\frac{\partial  u}{\partial x}(t,x)\right|\leq \|f'\|_\infty 
\Ee \left( J_{T-t}^x(0) \right) \leq  C \|f'\|_\infty $. 
We introduce the probability $\Qq^\frac{1}{2}$ such that
$\frac{ d\Qq^\frac{1}{2}}{ d\Pp}\bigg|_{\FF_t}
=\frac{1}{\ZZ_t^{(0,\frac{1}{2})}}$. Denoting by $\Ee^\frac{1}{2}$ the
expectation under the probability 
$\Qq^\frac{1}{2}$, we have
\begin{align*} 
\frac{\partial  u}{\partial  x}(t,x) = \Ee^\frac{1}{2}
\left(f'(X_{T-t}^x(0))\ZZ_{T-t}^{(0,\frac{1}{2})}J_{T-t}^x(0)
\right).
\end{align*}	
From~\eqref{derivee-flot}, as 
$W_t = B^{\frac{1}{2}}_t +\int _0^t \frac{\sigma }{2\sqrt {X^x_s(0)}} ds$, we notice
that 
\begin{align*} 
J^x_t(0) = \exp\left( \int _0^t b'(X_s^x(0) ds 
+ \frac{\sigma  }{2}\int _0^t \frac{dB^{\frac{1}{2}}_s}{\sqrt {X^x_s(0)}}
+ \frac{\sigma ^2}{8} \int _0^t \frac{ds}{X^x_s(0)}\right) 
\end{align*} 
and that $\ZZ_{T-t}^{(0,\frac{1}{2})}J_{T-t}^x(0) = \exp\left(\int _0^{T-t}
b'(X^x_s(0)) ds \right)$, from the definition of
$\ZZ_t^{(0,\frac{1}{2})}$ in \eqref{def_Z}. Hence, 
$\frac{\partial  u}{\partial  x}(t,x) = \Ee^\frac{1}{2}
\left[f'(X_{T-t}^x(0)) \exp\left(\int _0^{T-t} b'(X^x_s(0)) ds
\right)\right]$. 
As $\LL^{\Qq^{\frac{1}{2}}}(X_\cdot^x(0)) =
\LL^{\Pp}(X^x_\cdot(\frac{1}{2}))$, for $x>0$, we
finally obtain the following expression for $\frac{\partial  u}{\partial  x}(t,x)$: 
\begin{align}\label{new-expression-derivee-1-en-x}
\frac{\partial  u}{\partial  x}(t,x) = \Ee\left[ f'(X_{T-t}^x(\frac{1}{2}))
\exp\left(\int _0^{T-t} b'(X^x_s(\frac{1}{2})) ds\right)\right].
\end{align}
Now the right-hand side of \eqref{new-expression-derivee-1-en-x} is a
continuous function on $[0,T]\times  [0,+\infty )$ so that $u\in C^{1,1}([0,T]\times 
[0,+\infty ))$. Moreover for $x>0$, by  Proposition~\ref{justification-cir},  $\frac{\partial  u}{\partial  x}(t,x)$ is continuously differentiable  and 
\begin{align}\label{expression-derivee-2-en-x}
\begin{array}{r}
\frac{\partial  ^2u}{\partial  x^2}(t,x) = 
\Ee\left[f''(X_{T-t}^x(\frac{1}{2}))J_{T-t}^x(\frac{1}{2})\exp\left(\int _0^{T-t}
b'(X^x_s(\frac{1}{2})) ds\right)\right] \\
 + \Ee\left[f'(X_{T-t}^x(\frac{1}{2}))
\exp\left(\int _0^{T-t} b'(X^x_s(\frac{1}{2})) ds\right)
\left(\int _0^{T-t} b''(X^x_s(\frac{1}{2})) J_{s}^x(\frac{1}{2})
ds\right)\right].
\end{array}
\end{align}
As previously, we can conclude that $\left| \frac{\partial  ^2u}{\partial  x^2}(t,x)\right|$ is bounded
uniformly in $x$. In order to
obtain an expression for $\frac{\partial  ^2u}{\partial  x^2}(t,x)$ continuous in
$[0,T]\times [0,+\infty )$ and also to compute the third derivative, 
we need to transform the expression in~\eqref{expression-derivee-2-en-x} in order to 
avoid again the appearance of the derivative of $J^x_t(\frac{1}{2})$ that we 
do not control. Thanks to the Markov property and  the time homogeneity
 of the process
$X^x_t(\frac{1}{2}))$, for any $s \in [0, T-t]$, 
\begin{align*}
& \Ee\left. \left[ f'(X_{T-t}^x(\frac{1}{2})) \exp\left(\int _s^{T-t}
b'(X_u^x(\frac{1}{2}))\right)  \right/ \FF_s\right] \\
& = \Ee \left[ f'(X_{T-t-s}^{y}(\frac{1}{2})) \exp\left(\int _0^{T-t-s}
b'(X_u^{y}(\frac{1}{2}))\right)\right]\bigg|_{y =
X^x_s(\frac{1}{2}))}. 
\end{align*}
By using~\eqref{new-expression-derivee-1-en-x}, we get
$\Ee [ f'(X_{T-t}^x(\frac{1}{2})) \exp(\int _s^{T-t}
b'(X_u^x(\frac{1}{2}))) / \FF_s]
= \frac{\partial  u}{\partial  x}(t+s,X^x_s(\frac{1}{2})))$.
We introduce this last equality in the second term of the right-hand
side of~\eqref{expression-derivee-2-en-x}: 
\begin{align*}
& \Ee\left[f'(X_{T-t}^x(\frac{1}{2}))
\exp\left(\int _0^{T-t} b'(X^x_u(\frac{1}{2})) du\right)
\left(\int _0^{T-t} b''(X^x_s(\frac{1}{2})) J_{s}^x(\frac{1}{2})
ds\right)\right] \\
& = \Ee \left[ \int _0^{T-t}
\Ee \left. \left(f'(X_{T-t}^x(\frac{1}{2}))\exp\left(\int _s^{T-t}
b'(X^x_u(\frac{1}{2})) du\right) \right/ \FF_s\right) \right.\\
& \hspace{3cm}\left. 
\times  \exp\left(\int _0^{s} b'(X^x_u(\frac{1}{2})) du\right)
b''(X^x_s(\frac{1}{2}))J_{s}^x(\frac{1}{2})ds \right] \\
& =  \int _0^{T-t}\Ee \left[
\frac{\partial  u}{\partial  x}(t+s,X^x_s(\frac{1}{2}))) 
\exp\left(\int _0^{s} b'(X^x_u(\frac{1}{2})) du\right)
b''(X^x_s(\frac{1}{2}))J_{s}^x(\frac{1}{2})\right] ds. 
\end{align*}
Coming back to~\eqref{expression-derivee-2-en-x}, this leads to the
following expression  for $\frac{\partial  ^2u}{\partial  x^2}(t,x)$:
\begin{align*}
\frac{\partial  ^2u}{\partial  x^2}(t,x) = 
\Ee\left[f''(X_{T-t}^x(\frac{1}{2}))J_{T-t}^x(\frac{1}{2})\exp\left(\int _0^{T-t}
b'(X^x_s(\frac{1}{2})) ds\right)\right] \\
 + \int _0^{T-t}\Ee \left[
\frac{\partial  u}{\partial  x}(t+s,X^x_s(\frac{1}{2}))) 
\exp\left(\int _0^{s} b'(X^x_u(\frac{1}{2})) du\right)
b''(X^x_s(\frac{1}{2}))J_{s}^x(\frac{1}{2})\right] ds. 
\end{align*}
We introduce the probability $\Qq^1$ such that
$\frac{ d\Qq^1}{
d\Pp}\bigg|_{\FF_t}=\frac{1}{\ZZ_t^{(\frac{1}{2},1)}}$. Then,
\begin{align*}
\frac{\partial  ^2u}{\partial  x^2}(t,x) = 
\Ee^1\left[\ZZ_{T-t}^{(\frac{1}{2},1)}f''(X_{T-t}^x(\frac{1}{2}))J_{T-t}^x(\frac{1}{2})\exp\left(\int _0^{T-t}
b'(X^x_s(\frac{1}{2})) ds\right)\right] \\
 + \int _0^{T-t}\Ee^1 \left[\ZZ_s^{(\frac{1}{2},1)}
\frac{\partial  u}{\partial  x}(t+s,X^x_s(\frac{1}{2}))) 
\exp\left(\int _0^{s} b'(X^x_u(\frac{1}{2})) du\right)
b''(X^x_s(\frac{1}{2}))J_{s}^x(\frac{1}{2})\right] ds. 
\end{align*}
Again for all $\theta \in [0,T]$, we have that
$\ZZ_\theta  ^{(\frac{1}{2},1)}J_{\theta  }^x(\frac{1}{2}) = \exp\left(\int _0^{\theta  }
b'(X^x_u(\frac{1}{2})) du\right)$ 
and
\begin{align*}
\frac{\partial  ^2u}{\partial  x^2}(t,x) = & 
\Ee^1\left[f''(X_{T-t}^x(\frac{1}{2}))\exp\left(2\int _0^{T-t}
b'(X^x_s(\frac{1}{2})) ds\right)\right] \\
& + \int _0^{T-t}\Ee^1 \left[
\frac{\partial  u}{\partial  x}(t+s,X^x_s(\frac{1}{2}))) 
\exp\left(2\int _0^{s} b'(X^x_u(\frac{1}{2})) du\right)
b''(X^x_s(\frac{1}{2}))\right] ds. 
\end{align*}
As $\LL^{\Qq^{1}}(X^x_.(\frac{1}{2})) = \LL^{\Pp}(X^x_.(1))$, we
finally obtain the following expression for $\frac{\partial  ^2u}{\partial  x^2}(t,x)$: 
\begin{align}\label{new-expression-derivee-2-en-x}
\begin{array}{ll}
\frac{\partial  ^2u}{\partial  x^2}(t,x)=& 
\Ee\left[f''(X_{T-t}^x(1))\exp\left(2\int _0^{T-t}
b'(X^x_s(1)) ds\right)\right] \\
&+ \int _0^{T-t}\Ee\left[
\frac{\partial  u}{\partial  x}(t+s,X^x_s(1))) 
\exp\left(2\int _0^{s} b'(X^x_u(1)) du\right)
b''(X^x_s(1))\right] ds 
\end{array}
\end{align}
from which, we deduce that $u\in C^{1,2}([0,T]\times  [0,+\infty ))$.  
As $J^x_s(1)= \frac{dX^x_s(1)}{dx}$ exists and is given
by~\eqref{derivee-flot}, for $x>0$, 
$\frac{\partial ^2u}{\partial  x^2}(t,x)$ is continuously differentiable (see again
Proposition~\ref{justification-cir}) and
\begin{align}\label{expression-derivee-3-en-x}
\begin{array}{ll}
\frac{\partial  ^3u}{\partial  x^3}(t,x) = &
\Ee\left\{\exp\left(2\int _0^{T-t}b'(X^x_s(1)) ds\right)\right.  \\
& \hspace{0.6cm}\left. \times  \left[ f^{(3)}(X^x_{T-t}(1))J^x_{T-t}(1) +
2f''(X^x_{T-t}(1)) \int _0^{T-t} b''(X^x_s(1)) J^x_s(1)ds \right]\right\} \\
& + \int _0^{T-t}\Ee\left\{
\exp\left(2\int _0^s b'(X^x_u(1)du\right)\right. \\
& \hspace{0.6cm}\left. \times  \left[ J^x_s(1)
\left(\frac{\partial  ^2u}{\partial  x^2}(t+s,X^x_s(1))) b''(X^x_s(1)) +
\frac{\partial  u}{\partial  x}(t+s,X^x_s(1))) b^{(3)}(X^x_s(1))\right)
\right. \right. \\
& \left. \left. \hspace{2cm}
+ 2 \frac{\partial  u}{\partial  x}(t+s,X^x_s(1)))b''(X^x_s(1))\int _0^s
b''(X^x_u(1))J^x_u(1) du\right]\right\}ds.
\end{array}	
\end{align}
By Lemma~\ref{Martingale-expo-lambda}, the derivatives of $f$ and $b$
being bounded up to the order 3, we get immediately that $|\frac{\partial 
^3u}{\partial  x^3}(t,x)|\leq C$ uniformly in $x$.

The computation of the fourth  derivative uses similar arguments. 
We detail it in the Appendix~\ref{end_proof_deriv4}. 

In view of~\eqref{borne_u_alpha=1/2}
and~\eqref{borne_deriv_u_alpha=1/2}, one can  adapt easily the proof
of the Theorem~6.1 in
\cite{friedman-75} and show that $u(t,x)$ solves the
Cauchy problem~\eqref{edp_u_alpha=1/2}. 
\end{proof}

\subsubsection{On the approximation process}
According to~\eqref{schema} and~\eqref{schema-continu-bis},
the discrete time process $(\oX)$ associated to $(X)$ is 
\begin{align}\label{schema-CIR-gene}
\left\{
\begin{array}{l}
\oX_0 = x_0,\\
\oX_{t_{k+1}}=\left| \oX_{t_k} + b(\oX_{t_k})\Dt +
\sigma  \sqrt {\oX_{t_k}} (W_{t_{k+1}}-W_{t_k})\right|,
\end{array} \right.
\end{align} 
$k=0,\ldots ,N-1$, and the time continuous version $(\oX_{t},0\leq t\leq T)$ satisfies 
\begin{align}\label{schema-continu-bis-CIR-gene}
\oX_t =  x_0 + \int _0^t \sgn(\oZ_s) b(\oX_{\eta (s)})ds
+ \sigma  \int _0^t \sgn(\oZ_s)\sqrt {\oX_{\eta (s)}} dW_s + \frac{1}{2}L^0_t(\oX), 
\end{align}  
where $\sgn(x) = 1-2~\ind_{(x\leq 0)}$, and for any $t\in [0,T]$, 
\begin{align}\label{Z_t-CIR-gene}
\oZ_t = \oX_{\eta (t)} + (t-\eta (t))b(\oX_{\eta (t)})
+ \sigma \sqrt {\oX_{\eta (t)}} (W_t-W_{\eta (t)}). 
\end{align} 
In this section, we are interested in the behavior of the processes
$(\oX)$ and $(\oZ)$ visiting the point $0$. The main result is the
following
\begin{proposition}\label{proposition-schema-CIR-gene-tps-local}
Let $\alpha  = \frac{1}{2}$.
Assume {\rm (H1)}. For $\Dt$ sufficiently  small ($\Dt \leq 1/ (2K)$),
there exists a constant $C >0$, depending on $b(0)$, $K$, $\sigma $, $x_0$
and $T$ but not in $\Dt$, such that 
\begin{align}  \label{UP-local-time}
\Ee\left.\left(L^0_t(\oX) -L^0_{\eta (t)}(\oX)\right/\FF_{\eta (t)}\right) &\leq
C \Dt \exp\left(- \frac{\oX_{\eta (t)}}{16\sigma  ^2\Dt}\right) \nonumber \\ 
\mbox{and } ~
\Ee L^0_T(\oX) &\leq
C \left(\frac{\Delta  t}{x_0}\right)^{\frac{b(0)}{\sigma  ^2}}. 
\end{align} 
\end{proposition}
The  upper bounds above, for the local time $(L^0_\cdot  (\oX))$,  are based on the
following technical lemmas:
\begin{lemma}\label{lemma-aux-CIR-gene}
Assume {\rm (H1)}. Assume also that $\Dt$ is
sufficiently  small ($\Dt \leq 1/ (2K)\wedge  x_0)$. Then  for any
$\gamma  \geq 1$, there exists a positive constant $C$,  
depending on all the parameters  $b(0)$, $K$, $\sigma $, $x_0$, $T$ and
also on $\gamma $, such that 
\begin{align*} 
\sup_{k=0,\ldots , N} \Ee\exp\left(-\frac{\oX_{t_k}}{\gamma  \sigma  ^2\Delta 
t}\right)\leq C
\left(\frac{\Dt}{x_0}\right)^{\frac{2b(0)}{\sigma  ^2}(1 - \frac{1}{2\gamma })}.
\end{align*} 
\end{lemma} 
\begin{lemma}\label{lemma-schema-CIR-gene-1}
Assume {\rm (H1)}. For $\Dt$ sufficiently  small ($\Dt \leq 1/ (2K)$), for
any $t \in [0,T]$,
\begin{align*}
\Pp\left.\left( \oZ_t\leq 0 \right/ \oX_{\eta (t)} \right) \leq
\frac{1}{2} 
\exp\left(-\frac{\oX_{\eta (t)}}{2 (1 - K \Dt)^{-2} ~\sigma  ^2 \Dt}\right).
\end{align*} 
\end{lemma}
As $2(1 - K \Dt)^{-2} > 1$ when $\Dt \leq 1/(2K)$, the combination of Lemmas
\ref{lemma-aux-CIR-gene} and~\ref{lemma-schema-CIR-gene-1} leads to 
\begin{align}\label{VNcir}
\Pp \left(\oZ_t\leq 0\right) \leq C
\left(\frac{\Dt}{x_0}\right)^{\frac{b(0)}{\sigma  ^2}}.
\end{align} 
We give successively the proofs of Lemmas~\ref{lemma-aux-CIR-gene},
\ref{lemma-schema-CIR-gene-1} and
Proposition~\ref{proposition-schema-CIR-gene-tps-local}. 

\begin{proof}[Proof of Lemma~\ref{lemma-aux-CIR-gene}]
First, we show that there exits a positive sequence
$(\mu _j,0\leq j\leq N)$ such that, for any $k \in \{1,\ldots  , N\}$, 
\begin{align*}
\Ee\exp\left(-\frac{\oX_{t_k}}{\gamma  \sigma  ^2\Dt}\right) \leq
\exp\left(-b(0)\sum _{j=0}^{k-1}\mu _j \Dt\right)\exp\left(-x_0 \mu _k\right). 
\end{align*} 
We set $\mu _0=\frac{1}{\gamma   \sigma  ^2\Dt}$. 
By~\eqref{schema-CIR-gene}, as $-b(x)\leq -b(0)+Kx$, for all $x\in \er$,
we have that  
\begin{align*}
\Ee \exp\left(-\mu _0 \oX_{t_k}\right) 
\leq \Ee\exp\left(-\mu _0\left(\oX_{t_{k-1}}+(b(0)-K\oX_{t_{k-1}})
\Dt+\sigma  \sqrt {\oX_{t_{k-1}}}\Delta  W_{t_{k}}\right)\right). 
\end{align*} 
$\Delta  W_{t_{k}}$ and $\oX_{t_{k-1}}$ being independent,
$ \Ee\exp(-\mu _0 \sigma  \sqrt {\oX_{t_{k-1}}}\Delta  W_{t_{k}}) =
\Ee\exp(\frac{\sigma ^2}{2}\mu _0^2 \Dt\oX_{t_{k-1}})$. 
Thus
\begin{align*} 
\Ee\exp\left(-\mu _0\oX_{t_k}\right) \leq &
\exp\left(-\mu _0 b(0)\Dt \right) 
\Ee\exp\left(-\mu _0\oX_{t_{k-1}}
\left(1-K\Dt-\frac{\sigma ^2}{2}\mu _0 \Dt\right)\right)\\
 = & \exp\left(-\mu _0 b(0)\Dt \right)
\Ee\exp\left(-\mu _1\oX_{t_{k-1}}\right), 
\end{align*} 
where we set $\mu _1 = \mu _0 ( 1-K\Dt-\frac{\sigma ^2}{2}\mu _0\Dt)$.
Consider now the sequence $(\mu _j)_{i\in \nat}$ defined by
\begin{align}\label{sequence}
\begin{cases}
\mu _0=\frac{1}{\gamma  \sigma  ^2\Delta  t},\\
\mu _{j}=\mu _{j-1}\left(1-K\Dt-\frac{\sigma ^2}{2} \mu _{j-1}\Dt\right),~j\geq 1. 
\end{cases}
\end{align} 
An easy computation shows that if $\gamma  \geq 1$ and $\Dt \leq \frac{1}{2K}$, then
$(\mu _j)$ is a positive and  decreasing sequence.
For any $j\in \{0,\ldots ,k-1\}$, by the same computation we have  
\begin{align*} 
\Ee\exp\left(-\mu _j\oX_{t_{k-j}}\right) \leq
\exp\left(-b(0)\mu _j \Dt \right)
\Ee\exp\left(-\mu _{j+1}\oX_{t_{k-j-1}}\right)
\end{align*} 
and it follows by induction that 
\begin{align}\label{aux_1}
\Ee\exp\left(- \mu _0\overline {X}_{t_k}\right)
\leq \exp\left(-b(0)\sum _{j=0}^{k-1}\mu _j \Dt \right)
\exp\left(-x_0\mu _{k}\right).
\end{align} 
Now, we study the sequence $(\mu _j,0\leq j\leq N)$. 
For any $\alpha >0$, we consider the non-decreasing function
$f_{\alpha  }(x) :=\frac{x}{1+\alpha  x}$, $x \in \er$. We note that
$(f_{\alpha }\circ f_{\beta  })(x) = f_{\alpha  +\beta  }(x)$. Then, for any $j\geq 1$, 
the sequence $(\mu _j)$ being  decreasing, $\mu _j \leq \mu _{j-1} - \frac{\sigma 
^2}{2}\Dt \mu _j \mu _{j-1}$, and 
\begin{align}\label{aux_2}
\mu _j \leq f_{\frac{\sigma  ^2}{2}\Dt}(\mu _{j-1})\leq
f_{\frac{\sigma  ^2}{2}\Dt}\left( f_{\frac{\sigma  ^2}{2}\Dt}(\mu _{j-2})\right)\leq 
\ldots  \leq
f_{\frac{\sigma  ^2}{2}j\Dt}(\mu _0).
\end{align} 
The next step consists in proving, by induction, the following  lower
bound for the $\mu _j$: 
\begin{align}\label{aux_3}
\mu _{j} \geq \mu _1 \left(\frac{1}{1 + \frac{\sigma  ^2}{2}\Dt(j-1)\mu _0}\right)
- K \left(\frac{\Dt(j-1)\mu _0 }{1 + \frac{\sigma  ^2}{2}\Dt(j-1)\mu _0}\right), ~\forall  
j\geq1.
\end{align} 
\eqref{aux_3} is clearly true for $j=1$. Suppose~\eqref{aux_3} holds
for $j$. By~\eqref{sequence} and~\eqref{aux_2} 
\begin{align*}
\mu _{j+1} = \mu _{j}\left(1-\frac{\sigma ^2}{2}\Dt \mu _{j}\right) - K\Dt \mu _{j} 
\geq & \mu _{j}\left(1-\frac{\sigma ^2}{2}\Dt f_{\frac{\sigma  ^2}{2}j\Dt}(\mu _0)\right)
- K\Dt f_{\frac{\sigma  ^2}{2}j\Dt}(\mu _0)\\
\geq & \mu _{j}\left(\frac{1 + \frac{\sigma ^2}{2}\Dt (j-1) \mu _0}
{1 + \frac{\sigma ^2}{2}\Dt j \mu _0}\right)
- K\left(\frac{\Dt \mu _0}{1 + \frac{\sigma ^2}{2}\Dt j \mu _0}\right)
\end{align*} and we conclude by using~\eqref{aux_3} for $\mu _j$. Now,
we replace $\mu _0$ by its value $\frac{1}{\gamma  \sigma  ^2 \Dt}$ in~\eqref{aux_3}
and obtain that $\mu _{j} \geq \frac{2\gamma  -1}{\Dt \gamma \sigma  ^2 (2\gamma  -1 +j)} -
\frac{2 K}{\sigma  ^2}$, for any $j\geq 0$. Hence, 
\begin{align*}
\sum _{j=0}^{k-1}\Dt \mu _j
& \geq \frac{1}{\gamma  \sigma  ^2} \sum _{j=0}^{k-1}\frac{2\gamma  -1}{2\gamma  -1 + j}
- \frac{2 K t_k}{\sigma  ^2}
\geq \frac{1}{\gamma  \sigma  ^2}\int _0^k \frac{2\gamma  -1}{2\gamma  -1 +u} du
- \frac{2 K T}{\sigma ^2}\\
& \geq \frac{2\gamma  -1}{\gamma  \sigma  ^2}  \ln \left(\frac{2\gamma  -1 + k}{2\gamma  -1}\right) - 
\frac{2KT}{\sigma ^2}. 
\end{align*} 
Coming back to~\eqref{aux_1}, we obtain that 
\begin{align*}
\Ee\exp\left(-\frac{\overline {X}_{t_k}}{\gamma \sigma  ^2\Delta  t}\right)
\leq & \exp\left( b(0) \frac{2KT}{\sigma ^2}\right)
\exp\left(x_0 \frac{2 K}{\sigma ^2}\right) \\
& \times  \left(\frac{2\gamma  -1}{2\gamma  -1 +k}\right)^{\frac{2b(0)}{\sigma  ^2}(1 -
\frac{1}{2\gamma  })}
\exp\left(-\frac{x_0}{\Dt \gamma  \sigma  ^2} \frac{(2\gamma  -1)}{(2\gamma  -1 +k)}\right). 
\end{align*}
Finally, we use the inequality $x^\alpha \exp(-x)\leq \alpha  ^\alpha  \exp(-\alpha )$, for all $\alpha >0$ and $x>0$. It
comes that 
\begin{align*}
\Ee\exp\left(-\frac{\overline {X}_{t_k}}{\gamma \sigma  ^2\Delta  t}\right)
\leq &\exp\left( b(0) \frac{2KT}{\sigma ^2}\right)
\exp\left(x_0 \frac{2 K}{\sigma ^2}\right) \\
& \times  \left(2 b(0)\frac{\Dt \gamma }{x_0}(1 - \frac{1}{2\gamma  })
\right)^{\frac{2b(0)}{\sigma  ^2}(1 - \frac{1}{2\gamma  })}
\exp\left(-\frac{2b(0)}{\sigma  ^2}(1 -\frac{1}{2\gamma })\right)\\
& \leq C 
\left(\frac{\Dt}{x_0}\right)^{\frac{2b(0)}{\sigma  ^2}(1 - \frac{1}{2\gamma 
})}
\end{align*}
where we set
\begin{align*} 
C 
= \left( b(0)(2\gamma  -1)\right)^{\frac{2b(0)}{\sigma  ^2}(1 -\frac{1}{2\gamma })} 
\exp\left( \frac{2}{\sigma ^2}\left( b(0) (KT - 1 + \frac{1}{2\gamma  })+ 
x_0 K\right)\right). 
\end{align*} 
\end{proof}
\begin{proof}[Proof of Lemma~\ref{lemma-schema-CIR-gene-1}]
Under {\rm (H1)}, $b(x) \geq b(0)-Kx$, for $x\geq0$. Then, by the definition of $(\oZ)$ in~\eqref{Z_t-CIR-gene}, 
\begin{align*}
\Pp\left(\oZ_t\leq 0\right) 
\leq
\Pp\left( W_t-W_{\eta (t)} \leq
\frac{-\oX_{\eta (t)}(1-K(t-\eta (t)))-b(0)(t-\eta (t))}{\sigma \sqrt {
\oX_{\eta (t)}}}, \oX_{\eta (t)}>0\right).
\end{align*} 
By using the Gaussian inequality $\Pp(G\leq \beta ) \leq 1/2 \exp(- \beta ^2/2)$, 
for a standard Normal r.v. $G$ and $\beta  <0$, we get
\begin{align*}	
\Pp\left(\oZ_t\leq 0\right)\leq \frac{1}{2}
\Ee\left[
\exp\left(-\frac{(\oX_{\eta (t)}(1-K(t-\eta (t)))+b(0)(t-\eta (t)))^2}
{2\sigma  ^2(t-\eta (t))\oX_{\eta (t)}}\right)\ind_{\{\oX_{\eta (t)}>0\}}\right]
\end{align*}
from which we finally obtain that 
$\Pp\left(\oZ_t\leq 0\right) \leq \frac{1}{2}
\Ee[\exp(-\frac{\overline {X}_{\eta (t)}}{2(1 -K\Dt)^{-2}
\sigma ^2\Dt})]$. 
\end{proof}
\begin{proof}[Proof of Proposition~\ref{proposition-schema-CIR-gene-tps-local}]
By the occupation time formula, for any $t\in[t_k, t_{k+1})$, $0\leq
k\leq N$, for any bounded Borel-measurable 
function $\phi $,  $\Pp$ a.s
\begin{align*}
\int _{\er} \phi(z)\left(L_t^z(\oX) -L_{t_k}^z(\oX)\right)dz
& = \int _{\er}\phi(z)\left(L_t^z(\oZ) -L_{t_k}^z(\oZ)\right)dz \\
&= \int _{t_k}^t \phi  (\oZ_s) d \langle\oZ,\oZ\rangle_s
 = \sigma ^2\int _{t_k}^t\phi(\oZ_s)\oX_{t_k}ds.
\end{align*}
Hence, for any $x>0$, an easy computation shows that 
\begin{align*}
& \int_{\er} \phi(z)\Ee\left.\left( L_t^z(\oX) -L_{t_k}^z(\oX)
\right/ \left\{\oX_{t_k} = x \right\}\right)dz = \sigma ^2\int _{t_k}^t x
\Ee\left.\left(\phi (\oZ_s) \right/ \left\{\oX_{t_k} = x \right\}\right)ds \\
& = \sigma  \int_{\er} \phi (z) \int _{t_k}^t  \frac{\sqrt {x} }{\sqrt {2\pi (s-t_k)}}
\exp\left(-\frac{(z - x -b(x)(s -t_k))^2}
{2 \sigma  ^2x(s-t_k)}\right)  ds~ dz. 
\end{align*}
Then, for any $z\in \er$, 
\begin{align*}
\Ee\left.\left(L_t^z(\oX) -L_{t_k}^z(\oX)
\right/\left\{\oX_{t_k} = x \right\}\right) =
\sigma  \int _{t_k}^t \frac{\sqrt {x} }{\sqrt {2\pi (s-t_k)}}
\exp\left(-\frac{(z - x  -b(x)(s -t_k))^2}
{2 \sigma  ^2x~(s-t_k)}\right) ds. 
\end{align*}
In particular for $z=0$ and $t=t_{k+1}$, 
\begin{align*}
\Ee\left.\left(L_{t_{k+1}}^0(\oX) -L_{t_k}^0(\oX)\right/
\left\{\oX_{t_k} = x \right\}\right)
 = \sigma  \int _{0}^{\Dt} \frac{\sqrt {x} }{\sqrt {2\pi  s}}
\exp\left(-\frac{(x +  b(x) s)^2}
{2 \sigma  ^2x~ s }\right) ds. 
\end{align*}
From {\rm (H1)},  $b(x) \geq - K x$, with  $K\geq 0$. Then,
\begin{align*}
\Ee\left.\left(L_{t_{k+1}}^0(\oX) -L_{t_k}^0(\oX)\right/\FF_{t_k}\right) \leq 
\sigma  \int _{0}^{\Dt} \frac{\sqrt {\oX_{t_k}} }{\sqrt {2\pi  s}}
\exp\left(-\frac{\oX_{t_k}( 1 -Ks)^2 }
{2 \sigma  ^2 s }\right) ds. 
\end{align*}
For $\Dt$ sufficiently small, $1 -K\Dt\geq 1/2$ and
\begin{align*} 
\Ee\left.\left(L_{t_{k+1}}^0(\oX) -L_{t_k}^0(\oX)\right/\FF_{t_k}\right)  \leq 
\sigma  \int _{0}^{\Dt}\frac{\sqrt {\oX_{t_k}} }{\sqrt {2\pi  s}}
\exp\left(-\frac{\oX_{t_k} }
{8 \sigma  ^2 \Dt }\right) ds. 
\end{align*}
Now we use the upper-bound $a \exp(-\frac{a^2}{2})\leq 1, \forall  a \in\er$, 
to get 
\begin{align*}
\Ee\left(L_{t_{k+1}}^0(\oX) -L_{t_k}^0(\oX)\right)  & \leq 
\sigma ^2 \int _{0}^{\Dt} \frac{2\sqrt {\Dt}}{\sqrt {\pi  s}} \Ee\left[
\exp\left(-\frac{\oX_{t_k} }
{16 \sigma  ^2 \Dt }\right)\right] ds \\
& \leq \frac{4 \sigma  ^2 \Dt}{\sqrt {\pi }}
\sup_{k=0,\ldots ,N} \Ee \exp\left(-\frac{\oX_{t_k} }
{16 \sigma  ^2 \Dt }\right). 
\end{align*} 
We sum over $k$ and  apply the Lemma~\ref{lemma-aux-CIR-gene} to end the proof.  
\end{proof}
\subsection{Proof of Theorem~\ref{theorem-faible-CIR-gene} }
We are now in position to prove the Theorem~\ref{theorem-faible-CIR-gene}.  
To study the weak error $\Ee f(X_T) - \Ee f(\oX_T)$, we use the
Feynman--Kac representation of the  
solution $u(t,x)$ to the Cauchy problem~\eqref{edp_u_alpha=1/2}, studied in
the Proposition~\ref{proposition_edp_alpha=1/2}: for all
$(t,x) \in [0,T]\times (0, +\infty )$, $\Ee f(X_{T-t}^x) = u(t,x)$. Thus the
weak error becomes
\begin{align*}
\Ee f(X_T) - \Ee f(\oX_T) = \Ee \left(u(0,x_0) - u(T,\oX_T)\right)
\end{align*}
with $(\oX)$ satisfying \eqref{schema-continu-bis-CIR-gene}. 
Applying the It\^o's formula a first time, we obtain  that 
\begin{align*}
&\Ee \left[u(0,x_0) - u(T,\oX_T)\right] \\
& = - \int _0^T \Ee \left[ \frac{\partial  u}{\partial  t}(s,\oX_s) +
\sgn(\oZ_s) b(\oX_{\eta (s)})\frac{\partial  u}{\partial  x}(s,\oX_s)
+ \frac{\sigma  ^2}{2}\oX_{\eta (s)}\frac{\partial  ^2u}{\partial  x^2}(s,\oX_s)\right] ds \\
& \;\;\;\;-  \Ee \int _0^T \sgn(\oZ_s)\sigma  \sqrt {\oX_{\eta (s)}}
\frac{\partial  u}{\partial x}(s,\oX_s) dW_s - \Ee \int _0^T \frac{1}{2}\frac{\partial  u}{\partial x}(s,\oX_s) dL^0(\oX)_s. 
\end{align*}
From  Proposition~\ref{proposition_edp_alpha=1/2} and  Lemma~\ref{majo-moments-oX},
we easily check that  the stochastic integral $(\int _0^\cdot \sgn(\oZ_s) \sqrt {\oX_{\eta (s)}}\frac{\partial  u}{\partial x}(s,\oX_s) dW_s)$
is a martingale. Furthermore, we use the Cauchy problem~\eqref{edp_u_alpha=1/2} to get   
\begin{align*}
&\Ee \left[u(0,x_0) - u(T,\oX_T)\right] \\
& =  - \int _0^T \Ee \left[ \left(b(\oX_{\eta (s)}) -b(\oX_s)\right)
\frac{\partial  u}{\partial  x}(s,\oX_s)
+ \frac{\sigma  ^2}{2}\left(\oX_{\eta (s)}-\oX_s\right)
\frac{\partial  ^2u}{\partial  x^2}(s,\oX_s)\right] ds \\
& \;\;\;\;- \Ee \int _0^T \frac{1}{2}\frac{\partial  u}{\partial x}(s,\oX_s) dL^0(\oX)_s 
+ \int _0^T 2  \Ee \left( \ind_{\{\oZ_s \leq 0\}} b(\oX_{\eta (s)})
\frac{\partial u}{\partial  x}(s,\oX_s)\right) ds. 
\end{align*}
From Proposition~\ref{proposition-schema-CIR-gene-tps-local},
\begin{align*}
\left| \Ee \int _0^T \frac{\partial  u}{\partial x}(s,\oX_s) dL^0(\oX)_s \right| \leq
\left\|\frac{\partial  u}{\partial x}\right\|_{\infty } \Ee
\left(L^0_T(\oX)\right)
\leq C \left( \frac{\Delta  t}{x_0}\right)^{\frac{b(0)}{\sigma  ^2}}. 
\end{align*}
On the other hand, by Lemma~\ref{lemma-schema-CIR-gene-1} for any $s\in [0,T]$,
\begin{align*}
\left|2\Ee \left( \ind_{\{\oZ_s \leq 0\}} b(\oX_{\eta (s)})
\frac{\partial u}{\partial  x}(s,\oX_s)\right)\right|
\leq 
\left\|\frac{\partial  u}{\partial x}\right\|_{\infty }
\Ee \left[ (b(0) + K \oX_{\eta (s)})
\exp\left( - \frac{\oX_{\eta (s)}}{8 \sigma  ^2 \Dt}\right)\right].
\end{align*}
As for any $x\geq 0$, $x\exp(-\frac{x}{16\sigma  ^2\Delta  t})\leq 16\sigma  ^2\Dt$,  we
conclude, by Lemma~\ref{lemma-aux-CIR-gene}, that 
\begin{align*}
\left|  \int _0^T 2  \Ee \left( \ind_{\{\oZ_s \leq 0\}} b(\oX_{\eta (s)})
\frac{\partial u}{\partial  x}(s,\oX_s)\right) ds \right| \leq C 
\left(\frac{\Dt}{x_0}\right)^{\frac{b(0)}{\sigma  ^2}}. 
\end{align*}
Hence,
\begin{align*}
&\left|\Ee \left[u(0,x_0) - u(T,\oX_T)\right]\right| \\
& \leq  \left|\int _0^T \Ee \left[ \left(b(\oX_{\eta (s)}) -b(\oX_s)\right)
\frac{\partial  u}{\partial  x}(s,\oX_s)
+ \frac{\sigma  ^2}{2}\left(\oX_{\eta (s)}-\oX_s\right)
\frac{\partial  ^2u}{\partial  x^2}(s,\oX_s)\right] ds \right| + C
\left(\frac{\Dt}{x_0}\right)^{\frac{b(0)}{\sigma  ^2}}.
\end{align*}	
By applying  the It\^o's formula a second time ($\frac{\partial  u}{\partial  x}$ is a
$C^3$ function with bounded derivatives), 
\begin{align*} 
&\Ee\left[ \left(b(\oX_s) -b(\oX_{\eta (s)}\right) \frac{\partial  u}{\partial 
x}(s,\oX_s)\right]  \\
 & = \Ee\int _{\eta (s)}^s \sgn(\oZ_\theta  ) b(\oX_{\eta (s)})\left[ \left(b(\oX_\theta  )
-b(\oX_{\eta (s)}\right) \frac{\partial ^2 u}{\partial x^2}(s,\oX_\theta  ) + b'(\oX_\theta  ) \frac{\partial  u}{\partial 
x}(s,\oX_\theta  )\right] d\theta   \\
 &+ \Ee\int _{\eta (s)}^s\frac{\sigma ^2}{2} \oX_{\eta (s)}\left[ \left(b(\oX_\theta )
-b(\oX_{\eta (s)}\right) \frac{\partial ^3 u}{\partial x^3}(s,\oX_\theta  )
+ 2b'(\oX_\theta  )\frac{\partial ^2u}{\partial x^2}(s,\oX_\theta  ) 
+ b''(\oX_\theta  ) \frac{\partial  u}{\partial x}(s,\oX_\theta  )\right] d\theta   \\
 &+ \Ee\int _{\eta (s)}^s\frac{1}{2} \left[ \left(b(0 )
-b(\oX_{\eta (s)}\right) \frac{\partial ^2 u}{\partial x^2}(s,0 ) + b'(0) \frac{\partial  u}{\partial 
x}(s,0 )\right]dL^0_\theta  (\oX)
\end{align*} 
so that
\begin{align*} 
&\left|\Ee\left[ \left(b(\oX_s) -b(\oX_{\eta (s)}\right)
\frac{\partial  u}{\partial x}(s,\oX_s)\right]\right|\\
&\leq C \Dt \left(1 + \sup_{0\leq \theta  \leq T} \Ee|\oX_\theta  |^2 + \Ee\left\{ (1 +
|\oX_{\eta (s)}|) \left(L^0_s(\oX) -L^0_{\eta (s)}(\oX)\right)\right\}\right)
\end{align*} 
and we conclude by Lemma~\ref{majo-moments-oX} and
Proposition~\ref{proposition-schema-CIR-gene-tps-local} that
\begin{align*}
\left|\Ee\left[ \left(b(\oX_s) -b(\oX_{\eta (s)}\right)
\frac{\partial  u}{\partial x}(s,\oX_s)\right]\right| \leq  C\left(\Dt +
\left(\frac{\Dt}{x_0}\right)^{\frac{b(0)}{\sigma ^2}}\right). 
\end{align*}
By similar arguments, we  show that
\begin{align*}
\left|\Ee\left[ \left(\oX_s -\oX_{\eta (s)}\right)
\frac{\partial ^2 u}{\partial x^2}(s,\oX_s)\right]\right| \leq  C\left(\Dt +
\left(\frac{\Dt}{x_0}\right)^{\frac{b(0)}{\sigma ^2}}\right)
\end{align*}
which ends the proof of Theorem \ref{theorem-faible-CIR-gene}.

\section {The case of processes with $1/2 < \alpha  <1$}\label{Cas2}
\subsection{Preliminary results}
In this section, $(X_t)$ denotes the solution of~\eqref{HUW_gene}
starting at $x_0$ at time 0 and 
$(X_t^{x})$, starting at $x\geq 0$ at time 0, is the 
unique strong solution to
\begin{align}\label{HUW_gene_flot}
X_t^{x} =  x + \int _{0}^t b(X_s^{x})ds
+\sigma  \int _{0}^t \left(X_s^{x}\right)^\alpha  dW_s. 
\end{align}
\subsubsection{On the exact solution}
We give some upper-bounds on inverse moments and exponential moments
of $(X_t)$. 
\begin{lemma}\label{moment-inverse-HUW-gene}
Assume {\rm (H1)}. Let $x>0$. For any $1/2 < \alpha  < 1$, for any $p>0$,
there exists a positive constant $C$, depending on the
parameters of the model \eqref{HUW_gene_flot} and on $p$ such that
\begin{align*}
\sup_{t\in [0,T]} \Ee\left[{\left(X^x_t\right)^{-p}}\right]
\leq C (1 + x^{-p}).
\end{align*}
\end{lemma}
\begin{proof}
Let $\tau _n$ be the stopping time defined by 
$\tau _n = \inf\{ 0<s\leq T; X_s^x \leq  1/n \}$. By the It\^o's formula, 
\begin{align*}
\Ee\left[{(X_{t\wedge  \tau  _n}^x)^{-p}}\right] = & {x^{-p}}
-p\Ee\left[ \int _0^{t\wedge  \tau  _n} \frac{b(X_s^x)ds}{(X^x_s)^{p+1}}\right]
+ p(p+1) \frac{\sigma  ^2}{2} \Ee\left[\int _0^{t\wedge  \tau _n} \frac{ds}{
(X_s^x)^{p+2(1-\alpha )}} \right]\\
\leq & {x^{-p}}
+ pK \int _0^{t}\Ee\left(\frac{1}{(X^x_{s\wedge  \tau  _n})^{p}}\right)ds \\
& + \Ee\left[\int _0^{t\wedge  \tau  _n} \left(p(p+1) \frac{\sigma  ^2}{2} \frac{1}{
(X_s^x)^{p+2(1-\alpha )}} - p \frac{b(0)}{(X^x_s)^{p+1}} \right) ds
\right]. 
\end{align*} 
It is possible to find a positive constant $C$ such
that, for any $x>0$,
\begin{align*}
\left(p(p+1) \frac{\sigma  ^2}{2} \frac{1}{
x^{p+2(1-\alpha )}} - p \frac{b(0)}{x^{p+1}} \right) \leq C. 
\end{align*}
An easy computation shows that
$\underline{C} = p \left(2\alpha -1\right)\frac{\sigma ^2}{2}
\left[ (p +2(1-\alpha ))\frac{\sigma  ^2}{2 b(0)} \right] ^{\frac{p+2(1-\alpha )}{2\alpha -1}}$ 
is the smallest one satisfying the
upper-bound above.  Hence,
\begin{align*}
\Ee\left[{(X_{t\wedge  \tau  _n}^x)^{-p}}\right] \leq  {x^{-p}} + 
\underline{C} T  + pK \int _0^t\sup_{\theta  \in
[0,s]}\Ee\left[{(X^x_{\theta  \wedge  \tau _n})^{-p}}\right]ds 
\end{align*} 
and by the Gronwall Lemma
\begin{align*}
\sup_{t\in [0,T]}\Ee\left[(X_{t\wedge  \tau  _n}^x)^{-p}\right]
\leq  \left(x^{-p} + \underline{C} T\right)\exp({pKT}).
\end{align*}
We end the proof, by taking the limit $n \rightarrow +\infty $. 
\end{proof} 
%
\begin{lemma}\label{minoration_1}
Assume {\rm (H1)}.  
\begin{description}
\item{(i) } For any $a\geq 0$, for all $0\leq t\leq T$,  a.s.
$(X^x_t)^{2(1-\alpha)}\geq {r}_t(a)$, where
$({r}_t(a),0\leq t\leq T)$ is the solution of the CIR Equation:
\begin{align*}  
 {r}_t(a)=x^{2(1-\alpha )}+\int _0^t(a-\lambda(a){r}_s(a))ds
+2\sigma (1-\alpha )\int _0^t\sqrt {{r}_s(a)}dW_s
\end{align*}
with
\end{description} 
\begin{align} 
\label{b} 
\lambda(a)=2(1-\alpha )K+\left(\frac{\left(2\alpha -1\right)^{2\alpha -1}\left(a+\sigma 
^2(1-\alpha )(2\alpha -1)\right)}{b(0)^{2(1-\alpha )}}\right)^{\frac{1}{2\alpha -1}}. 
\end{align}
\begin{description}
\item{(ii) }For all  $\mu \geq 0$, there exists a constant $C(T,\mu )>0$ with
a non decreasing  dependency on $T$ and $\mu $, depending also on $K$,
$b(0)$, $\sigma $, $\alpha $ and $x$ such that
\end{description}	
\begin{align}\label{novikov-hull-white-1}
\Ee\exp\left(\mu \int _0^T\frac{ds}{(X^x_s)^{2(1-\alpha )}}\right)\leq C(T,\mu ).
\end{align}
\begin{description}
\item{(iii) }The process $(M^x_t,0\leq t\leq T)$ defined by
\end{description}	
\begin{align}\label{M_t}
M^x_t=\exp\left(\alpha \sigma \int _0^t\frac{dW_s}{(X^x_s)^{1-\alpha }}-\alpha ^2\frac{\sigma 
^2}{2}\int _0^t\frac{ds}{(X^x_s)^{2(1-\alpha )}}\right) 
\end{align}
\begin{description}
\item{}
is a martingale. Moreover for all $p\geq 1$, there exists a positive
constant $C(T,p)$ depending also on $b(0)$, $\sigma $ and  $\alpha $, such that
\end{description}	
\begin{align}\label{moment-p-M}
\Ee\left(\sup_{t\in [0,\;T]}(M^x_t)^p\right)\leq C(T,p)
\left(1+\frac{1}{x^{\alpha  p}}\right).
\end{align}
\end{lemma}
\begin{proof} 
Let $Z_t=(X^x_t)^{2(1-\alpha )}$.  By the It\^o's  formula, 
\begin{align*}
Z_t=x^{2(1-\alpha )}+\int _0^t \beta  (Z_s) ds+2(1-\alpha )\sigma \int _0^t \sqrt {Z_s}dW_s, 
\end{align*}
where, for all $x>0$, the drift coefficient $\beta (x)$ is defined by 
\begin{align*}
\beta (x)=2(1-\alpha ){b(x^{\frac{1}{2(1-\alpha )}})}{x^{-\frac{2\alpha -1}{2(1-\alpha )}}}-\sigma 
^2(1-\alpha )(2\alpha -1). 
\end{align*}
From  {\rm (H1)}, $b(x)\geq b(0)- Kx$ and for all $x>0$, $\beta (x) \ge  \bar
\beta  (x)$, where we set 
\begin{align*}
\bar \beta (x) = 2(1-\alpha  ){b(0)}{x^{-\frac{2\alpha -1}{2(1-\alpha )}}}-2(1-\alpha )Kx-\sigma 
^2(1-\alpha )(2\alpha -1).
\end{align*}
For all $a\geq 0$ and  $\lambda(a)$ given by~\eqref{b}, we consider
$f(x) = \bar \beta (x) -  a + \lambda (a)x$. An easy computation shows that $f(x)$ 
has one minimum at the point $x^\star =
\left(\frac{b(0)(2\alpha -1)}{\lambda(a)-2(1-\alpha )K}\right)^{2(1-\alpha )}$.
Moreover, 
\begin{align*}
f(x^\star)
=\frac{b(0)^{2(1-\alpha )}}{(2\alpha -1)^{2\alpha -1}}\left(\lambda(a)-2(1-\alpha )K\right)^{2\alpha -1}-\left(a+\sigma 
^2(1-\alpha )(2\alpha -1)\right) 
\end{align*}
and when  $\lambda (a)$ is given by~\eqref{b}, $f(x^\star)=0$. We conclude
that $\beta  (x) \geq a -\lambda  (a) x$ and $(i)$ holds by the Comparison Theorem for 
the solutions of one-dimensional SDE. As a consequence, 
\begin{align*}
\Ee\exp\left(\mu \int _0^T\frac{ds}{(X^x_s)^{2(1-\alpha )}}\right)\leq
\Ee\left(\exp\left(\mu \int _0^T\frac{ds}{{r}_s(a)}\right)\right). 
\end{align*}
We want to apply the Lemma~\ref{novikov-cir}, on the exponential moments of 
the CIR process.  To this end, we must choose the constant 
$a$ such that $a \geq 4 (1-\alpha )^2\sigma  ^2$ and $\mu  \leq \frac{\nu ^2(a) (1-\alpha )^2 4 \sigma 
^2}{8}$, for $\nu (a)$ as in Lemma~\ref{novikov-cir}.
An easy computation shows that
$a = 4 (1-\alpha )^2\sigma  ^2\vee \left( 2(1-\alpha  )^2\sigma  ^2+(1-\alpha )\sigma  2\sqrt {2\mu }\right)$
is convenient and $(ii)$ follows by applying the Lemma~\ref{novikov-cir} to
the process $({r}_t(a),0\leq t\leq T)$. Thanks to  $(ii)$, the Novikov criteria applied to $M^x_t$ 
is clearly satisfied. 
Moreover,  by the integration by parts formula.
\begin{align*}
M^x_t
&=\left(\frac{X^x_t}{x}\right)^{\alpha }\exp\left(\int _0^t\left(-\alpha \frac{
b(X^x_s)}{X^x_s}+\alpha (1-\alpha )\frac{\sigma  ^2}{2}\frac{1}{(X^x_s)^{2(1-\alpha )}}\right)ds\right) 
\\ 
&\leq \left(\frac{X^x_t}{x}\right)^{\alpha }\exp(KT)\exp\left(\int _0^t\left(-\alpha \frac{ 
b(0)}{X^x_s}+\alpha (1-\alpha )\frac{\sigma  ^2}{2}\frac{1}{(X^x_s)^{2(1-\alpha )}}\right)ds\right).
\end{align*}
To end the proof, notice that it is possible to find a positive
constant $\lambda $ such that,  for any $x>0$, 
$-\frac{b(0)\alpha }{x}+\frac{\sigma 
^2\alpha (1-\alpha )}{2}\frac{1}{x^{2(1-\alpha )}}\leq \lambda  $. An easy
computation shows that
\begin{align*}
\underline{\lambda } =\frac{\alpha }{2}(2\alpha -1)\left[\frac{(1-\alpha )^{3-2\alpha }\sigma 
^2}{b(0)^{2(1-\alpha )}}\right]^{\frac{1}{2\alpha -1}}.
\end{align*}
is convenient.  Thus, 
$M^x_t\leq \left(\frac{X^x_t}{x}\right)^\alpha 
\exp\left((K+\underline{\lambda })T\right)$ and we
conclude by  using the Lemma~\ref{majo-moments-oX}. 
\end{proof}	

\subsubsection{On the associated Kolmogorov PDE} 
\begin{proposition}\label{proposition_edp_alpha>1/2}
Let $1/2 < \alpha  \leq 1$. Let $f$ be a $\er$-valued 
$C^4$ bounded function,  with bounded spatial derivatives up to the
order $4$. We consider the $\er$-valued function  defined on 
$[0,\,T]\times [0,+\infty )$ by $u(t,x)=\Ee f(X^x_{T-t})$.  
Then under  {\rm (H1)} and {\rm (H2)}, $u$ is in
$C^{1,4}([0,T]\times  (0,+\infty ))$ and there exists a positive constant  
$C$ depending on $f$, $b$ and $T$ such that
\begin{align*}
\left\|u \right\|_{L^\infty  \left([0,T]\times [0,+\infty )\right)} + 
\left\|\frac{\partial  u}{\partial  x}\right\|_{L^\infty  \left([0,T]\times 
[0,+\infty )\right)} &\leq C
\end{align*}
and for all $x>0$,
\begin{align*} 
\sup_{t\in[0,T]}\left|\frac{\partial  u}{\partial t}(t,x)\right| &\leq C(1 + x^{2\alpha }), \\
\mbox{ and } \sup_{t\in[0,T]} \sum _{k=2}^4 \left|\frac{\partial ^k u}{\partial 
x^k}\right|(t,x) &\leq C \left( 1 + \frac{1}{x^{q(\alpha )}}\right), 
\end{align*} 
where the constant $q(\alpha )>0$ depends only on $\alpha $. Moreover, $u(t,x)$ satisfies 
\begin{align}\label{edp_u_alpha>1/2}
\left\{\begin{array}{l}
\displaystyle  \frac{\partial  u}{\partial  t}(t,x) + b(x)\frac{\partial  u}{\partial  x}(t,x) +\frac{\sigma  ^2}{2}x^{2\alpha  }
\frac{\partial ^2u}{\partial  x^2}(t,x)=0,\;(t,x)\in [0,T]\times (0,+\infty  ), \\
\displaystyle  u(T,x)=f(x),\; x\in [0,+\infty  ).
\end{array} \right. 
\end{align}
\end{proposition}
The following Proposition~\ref{justification-hull-white} allows us to compute the derivatives of
$u(t,x)$. Equation~\eqref{HUW_gene_flot} has locally Lipschitz 
coefficients on $(0,+\infty )$,  
with locally Lipschitz first order
derivatives. Then $X_t^{x}$ is continuously differentiable and if
we denote $J^{x}_t = \frac{dX_t^{x}}{dx}$,  
the process $(J^{x}_t,0\leq t\leq T)$ satisfies the linear
equation
\begin{align}\label{J-epsilon-hull-white}
J^{x}_t = 1+ \int _{0}^t  J^{x}_s b'(X_s^{x})ds
+ \int _{0}^t \alpha \sigma   J^{x}_s \frac{dW_s}{(X_s^{x})^{1-\alpha }}. 
\end{align}
\begin{proposition}\label{justification-hull-white}
Assume {\rm (H1)} and {\rm (H2)}.   
Let $g(x)$, $h(x)$ and $k(x)$ be some  $C^1$ functions on $(0,+\infty )$ such
that, there exist $p_1>0$ and $p_2>0$, 
\begin{align*}
\begin{array}{ll}
\forall  x>0,~&|g(x)| + |g'(x)| + |h(x)| + |h'(x)| +  |k'(x)|
\leq C\left(1 + x^{p_1} + \frac{1}{x^{p_2}}\right), \\
&|k(x)| \leq C\left(1 + \frac{1}{x^{2(1-\alpha )}}\right).
\end{array}	
\end{align*} 
Let $v$  be the $\er$-valued function defined on $[0,T]\times (0,+\infty )$
by
\begin{align*} 
v(t,x) = \Ee \left[g(X_t^x) \exp(\int _0^t k(X_s^x) ds)\right] 
+ \int _0^t \Ee\left[h(X_s^x) \exp(\int _0^s k(X_\theta ^x ) d\theta  )\right] ds.
\end{align*} 
Then $v(t,x)$ is of class $C^1$ with
respect to $x$ and
\begin{align*} 
\frac{\partial  v}{\partial  x}(t,x)=& 
\Ee \left[\exp(\int _0^t k(X_s^x) ds)\left(g'(X_t^x)J^x_t +
g(X_t^x) \int _0^t k'(X_s^x)J^x_s ds\right)\right] \\
&+ \int _0^t \Ee\left[\exp(\int _0^s k(X_\theta ^x )
d\theta )\left(h'(X_s^x)J^x_s + h(X_s^x)\int _0^s k'(X_\theta ^x )J^x_\theta 
 d\theta \right)\right] ds.
\end{align*} 
\end{proposition}
The proof is postponed in the Appendix~\ref{proofs-justification}.
\begin{proof}[Proof of Proposition~\ref{proposition_edp_alpha>1/2}] 
Many arguments are similar to those of the proof of Proposition
\ref{proposition_edp_alpha=1/2}. Here, we restrict our attention on the main
difficulty which consists in obtaining  the upper bounds for the
spatial derivatives of $u(t,x)$ up to
the order 4. 
By Lemma~\ref{moment-inverse-HUW-gene}, for $x>0$, 
$(\int _0^t\frac{dW_s}{(X^x_s)^{1-\alpha }},0\leq t\leq T)$ is a locally square integrable
martingale. Then 
$J^{x}_t$ is given by  
\begin{align*}
J^{x}_t = \exp\left(\int _{0}^t b'(X_s^{x})ds +
\alpha  \sigma  \int _{0}^t\frac{dW_s}{(X_s^{x})^{1-\alpha }}
-\frac{\sigma ^2\alpha ^2 }{2}\int _{0}^t\frac{ds}{(X_s^{x})^{2(1-\alpha )}}\right) 
\end{align*}
Or equivalently $J^{x}_t  = \exp\left(\int  _0^tb'(X^x_s)ds\right)M_t$,
where $(M_t)$ is the martingale defined in~\eqref{M_t} and
satisfying~\eqref{moment-p-M}. Thus , we have
 $J^{x}_t  = \exp\left(\int  _0^tb'(X^x_s)ds\right)M_t$. $b'$ being
bounded, $\Ee J^x_t \leq \exp(KT)$ and for all $p> 1$, 
\begin{align}\label{moment-p-J1}
\Ee\left(\sup_{t\in [0,T]} (J_{t}^{x})^p\right)
\leq C(T)\left(1+\frac{1}{x^{\alpha p}}\right).
\end{align}
By Proposition~\ref{justification-hull-white}, $u(t,x)$ is differentiable and
\begin{align*}
\frac{\partial  u}{\partial x}(t,x)=\Ee \left[f'(X_{T-t}^{x})J_{T-t}^{x}\right].
\end{align*}
Then, $|\frac{\partial  u}{\partial x}(t,x)|\leq  \|f'\|_{\infty } \exp(KT)$. The 
integration by parts formula gives  
\begin{align*} 
J^x_t = \frac{(X^x_t)^\alpha  }{x^\alpha  }\exp\left( \int _0^t \left( b'(X_s^x)-\frac{\alpha 
b(X_s^x)}{X_s^x}+\frac{\sigma  ^2\alpha (1-\alpha )}{2(X_s^x)^{2(1-\alpha )}}\right)ds\right).
\end{align*}
We apply again the Proposition~\ref{justification-hull-white} to
compute $\frac{\partial  ^2u}{\partial  x^2}(t,x)$: for any $x>0$, 
\begin{align*}
\frac{d J^x_t}{d x} = -\frac{\alpha J^x_t }{x} +\frac{\alpha (J^x_t)^2 }{X^x_{t}} +
J^x_t \left( \int _0^{t}\left(b''(X_s^x)-\frac{\alpha  
b'(X_s^x)}{X_s^x}+\frac{\alpha  b(X_s^x)}{(X_s^x)^2}
-\frac{\sigma ^2\alpha (1-\alpha )^2}{(X_s^x)^{3-2\alpha }}\right)J_s^{x}ds\right)
\end{align*}
and 
\begin{align}\label{second-derivative}
\begin{array}{l}
\frac{\partial  ^2u}{\partial  x^2}(t,x)=\Ee\left[f''(X_{T-t}^x)(J_{T-t}^{x})^2\right]
-\frac{\alpha }{x}\frac{\partial  u}{\partial x}(t,x)
+ \alpha  \Ee\left[\frac{(J_{T-t}^{x})^2}{X_{T-t}^x}f'(X_{T-t}^x)\right]\\
+\Ee\left[f'(X_{T-t}^x)J_{T-t}^{x}\int _0^{T-t}\left(b''(X_s^x)-\frac{\alpha  
b'(X_s^x)}{X_s^x}+\frac{\alpha  b(X_s^x)}{(X_s^x)^2}
-\frac{\sigma ^2\alpha (1-\alpha )^2}{(X_s^x)^{3-2\alpha }}\right)J_s^{x}ds\right].
\end{array}	
\end{align}
By using the Cauchy-Schwarz Inequality
with Lemma~\ref{moment-inverse-HUW-gene} and estimate
\eqref{moment-p-J1}, the second term on the right-hand side is bounded 
by 
\begin{align*}
&\|f'\|_\infty  \Ee\left[\sup_{t\in [0,T]} (J_{t}^{x})^2\int _0^{T-t}
\left|b''(X_s^x)-\frac{\alpha  
b'(X_s^x)}{X_s^x}+\frac{\alpha  b(X_s^x)}{(X_s^x)^2}
-\frac{\sigma ^2\alpha (1-\alpha )^2}{(X_s^x)^{3-2\alpha }}\right| ds\right]\\
&\leq  C(T)\left(1+\frac{1}{x^{2(1+\alpha )}}\right).
\end{align*}
By using similar arguments, it comes that 
\begin{align*} 
\left|\frac{\partial  ^2u}{\partial  x^2}(t,x)\right|
\leq C(T)\left(1+\frac{1}{x^{2+2\alpha }}\right).
\end{align*}
We apply again the Proposition~\ref{justification-hull-white} to 
compute $\frac{\partial  ^3u}{\partial  x^3}(t,x)$ from~\eqref{second-derivative}  and
next $\frac{\partial  ^4u}{\partial  x^4}(t,x)$, the main difficulty being the
number of terms to write. In view of the expression of
$\frac{ d J^x_s}{d x}$, each term can be bounded by 
$C(T)(1 + x^{-2(n-1) -n\alpha })$, where $n$ is the derivation order, by using 
the Cauchy-Schwarz Inequality and the upper-bounds $\Ee\sup_{t\in [0,T]}
(J_{t}^{x})^p\leq C(T)(1+x^{-\alpha p})$ and
$\sup_{t\in [0,T]} \Ee(X^x_t)^{-p} \leq C (1 + x^{-p})$.
\end{proof}	

\subsubsection{On the approximation process}
When $1/2 < \alpha  < 1$, according to~\eqref{schema} and
\eqref{schema-continu-bis}, the discrete time process $(\oX)$ associated
to $(X)$ is 
\begin{align*}
\left\{
\begin{array}{l}
\oX_0 = x_0,\\
\oX_{t_{k+1}}=\left| \oX_{t_k} + b(\oX_{t_k})\Dt +
\sigma  \oX_{t_k}^\alpha   (W_{t_{k+1}}-W_{t_k})\right|,\;k=0,\ldots ,N-1, 
\end{array} \right.
\end{align*}
Its time continuous version
$(\oX_{t},0\leq t\leq T)$ satisfies 
\begin{align}\label{schema-continu-bis-HUL-gene}
\oX_t =  x_0 + \int _0^t \sgn(\oZ_s) b(\oX_{\eta (s)})ds
+ \sigma  \int _0^t \sgn(\oZ_s)\oX_{\eta (s)}^\alpha   dW_s + \frac{1}{2} L^0_t(\oX), 
\end{align} 
where for any $t\in [0,T]$, we set  
\begin{align}\label{Z_t-HUL-gene}
\oZ_t = \oX_{\eta (t)} + (t-\eta (t))b(\oX_{\eta (t)})
+ \sigma \oX_{\eta (t)}^\alpha   (W_t-W_{\eta (t)}),
\end{align}
so that, for all $t\in [0,\,T]$, $\oX_t=|\oZ_t|$. 

In the sequel, we will use the following notation:
\begin{align*}
\Oe(\Dt) = C(T)\exp\left(- \frac{C}{\Dt^{\alpha  -\frac{1}{2}}}\right),
\end{align*}
where the positive constants $C$ and $C(T)$ are independent of $\Dt$
but can depend on $\alpha  $, $\sigma  $ and $b(0)$. $C(T)$ is  non-decreasing in 
$T$. The quantity  $\Oe(\Dt)$ decreases exponentially fast  with
$\Dt$.

In this section, we are interested in the behavior of the processes
$(\oX)$ or $(\oZ)$ near $0$. We work under the
hypothesis {\rm (H3')}: $x_0 > \frac{b(0)}{\sqrt {2}}\Dt$.  We introduce the stopping
time $\tau $ defined by
\begin{align}\label{def_tD}
\tau  = \inf \left\{ s\geq 0; \oX_s < \frac{b(0)}{2}\Dt\right\}.
\end{align}
Under {\rm (H3')}, we are able to control probabilities like $\Pp(\tau 
\leq T)$. This is an important difference with the case $\alpha  =1/2$.
\begin{lemma}\label{stoppingtime-1}
Assume {\rm (H1), (H2)} and {\rm (H3')}. Then 
\begin{align}\label{stoppingtime}
\Pp\left(\tau  \leq T\right)\leq  \Oe(\Delta  t).
\end{align}
\end{lemma} 
\begin{proof}
The first step of the proof consists in obtaining the following
estimate:  
\begin{align}\label{oX-less}
\forall  k\in \{0,\ldots ,N\},~\Pp\left(\oX_{t_k}\leq \frac{b(0)}{\sqrt {2}}\Delta 
t\right)\leq \Oe(\Delta  t).
\end{align}
Indeed, as $b(x)\geq b(0)-Kx$ for $x \geq 0$, for $k\ge 1$, 
\begin{align*}
\begin{array}{ll}
& \Pp\left(\oX_{t_k}\leq \frac{b(0)}{\sqrt {2}}\Delta t\right)\\
& \leq \Pp\left(W_{t_k}-W_{t_{k-1}}
\leq\frac{-\oX_{t_{k-1}}(1-K\Delta  t)-b(0)(1-\frac{1}{\sqrt {2}})\Delta  t}{\sigma 
\oX_{t_{k-1}}^\alpha }, \oX_{t_{k-1}} > 0 \right).
\end{array}	
\end{align*}
As $\Dt$ is sufficiently small, 
by using the Gaussian inequality $\Pp(G\leq \beta ) \leq 1/2 \exp(- \beta ^2/2)$,  
for a standard Normal r.v. $G$ and $\beta  <0$, we get 
\begin{align*}
\begin{array}{ll}
&\Pp\left(\oX_{t_k}\leq \frac{b(0)}{\sqrt {2}}\Delta t\right)\\
&\leq \Ee\left[\exp\left(-\frac{\left(\oX_{t_{k-1}}(1-K\Delta t)
+b(0)(1-\frac{1}{\sqrt {2}})\Delta  t\right)^2}{2\sigma ^2\oX_{t_{k-1}}^{2\alpha }\Dt}\right)
\ind_{\{\oX_{t_{k-1}} > 0\}}\right] \\
&\leq \Ee\left[\exp\left(-\frac{\oX_{t_{k-1}}^{2(1-\alpha )}}
{8\sigma ^2\Dt}\right)\exp\left(-\frac{b(0)(1-\frac{1}{\sqrt {2}})}
{2\sigma  ^2\oX_{t_{k-1}}^{2\alpha -1}}\right)\ind_{\{\oX_{t_{k-1}} > 0\}}\right].
\end{array}	
\end{align*}
By separating the events $\left\{\oX_{t_{k-1}}\geq \sqrt {\Dt}\right\}$ and
$\left\{\oX_{t_{k-1}}<\sqrt {\Dt}\right\}$  in the expectation above, we
obtain 
\begin{align*}
\Pp\left(\oX_{t_k}\leq \frac{b(0)}{\sqrt {2}}\Delta 
t\right)\leq \exp\left(-\frac{1}{8\sigma 
^2\Dt^\alpha  }\right)+\exp\left(-\frac{b(0)(1-\frac{1}{\sqrt {2}})}{2\sigma 
^2(\Dt)^{\alpha -\frac{1}{2}}}\right)=\Oe(\Delta  t).
\end{align*}
Now we prove~\eqref{stoppingtime}. Notice that 
\begin{align*} 
\Pp\left(\tau  \leq T\right)
\leq \sum _{k=0}^{N-1}\Pp\left(\inf\limits_{t_k< s\leq  t_{k+1}}
\oZ_s\leq \frac{b(0)}{2}\Delta t~,~\oX_{t_{k}}> \frac{b(0)}{2}\Delta t\right). 
\end{align*} 
For each $k\in\{0,1,\ldots , N-1\}$, by using~\eqref{oX-less} and $b(x) \leq b(0) -K x$, we have
\begin{align*}
&\Pp\left(\inf\limits_{t_k< s\leq  t_{k+1}}\oZ_s\leq
\frac{b(0)}{2}\Delta t,~\oX_{t_{k}}> \frac{b(0)}{2}\Delta t\right)\\
&=\Pp\left(\inf\limits_{t_k< s\leq  t_{k+1}}\oZ_s\leq
\frac{b(0)}{2}\Delta t,~\oX_{t_{k}}>
\frac{b(0)}{2}\Delta t,~\oX_{t_k}\leq \frac{b(0)}{\sqrt {2}}\Dt\right)\\
&~~+\Pp\left(\inf\limits_{t_k< s\leq  t_{k+1}}\oZ_s \leq
\frac{b(0)}{2}\Delta t,~\oX_{t_k}> \frac{b(0)}{\sqrt {2}}\Dt\right)\\ 
&\leq \Pp\left(\oX_{t_k}\leq \frac{b(0)}{\sqrt {2}}\Dt\right)
+\Pp\left(\inf\limits_{t_k< s\leq  t_{k+1}}\oZ_s\leq
\frac{b(0)}{2}\Delta t,\oX_{t_k}> \frac{b(0)}{\sqrt {2}}\Dt\right)\\ 
&\leq \Oe(\Dt)
+\Ee \left\{\ind_{\left(\oX_{t_{k}}>\frac{b(0)}{\sqrt {2}}\Delta t\right)}\Pp\left(\inf\limits_{0< s\leq \Delta  t}\frac{x^{1-\alpha }}{\sigma }+\frac{b(0)-Kx}{\sigma  x^\alpha }s+
B_s\leq \frac{b(0)\Delta t}{2\sigma 
x^\alpha }\right)\left|_{x=\oX_{t_{k}}}\right.\right\}, 
\end{align*}
where $(B_t)$ denotes a Brownian motion independent of $(W_t)$. The
proof is ended if we show that  
\begin{align*} 
\Pp\left(\inf\limits_{0< s\leq \Delta  t}\frac{x^{1-\alpha }}{\sigma }
+\frac{(b(0)-Kx)}{\sigma  x^\alpha }s+
B_s\leq \frac{b(0)\Delta t}{2\sigma  x^\alpha }\right)\leq \Oe (\Dt),
\mbox{ for } x\geq \frac{b(0)}{\sqrt {2}}\Delta t.
\end{align*} 
We use the following formula (see \cite{borodin-salmien-96}): if
$(B^{\mu }_t,0\leq t\leq T)$ denotes a Brownian motion with drift $\mu $, starting
at $y_0$, then for all  $y\leq y_0$, 
\begin{align*}
\Pp\left(\inf\limits_{0<s<t}B^{\mu }_s\leq y\right)= 
\frac{1}{2}\text{erfc}\left(\frac{y_0-y}{\sqrt {2t}}+\frac{\mu \sqrt {t}}{\sqrt {2}}\right)
+\frac{1}{2}\exp(2\mu  (y-y_0))
\text{erfc}\left(\frac{y_0-y}{\sqrt {2t}}-\frac{\mu \sqrt {t}}{\sqrt {2}}\right), 
\end{align*}
where $\text{erfc}(z)=\frac{\sqrt {2}}{\sqrt {\pi }}\int _{\sqrt {2}z}^\infty  
\exp\left(-\frac{y^2}{2}\right)dy$, for all  $z\in \er$. We set
$\mu  = \frac{(b(0)-Kx)}{\sigma  x^\alpha }$, $y_0= \frac{x^{1-\alpha }}{\sigma }$ and we choose
$y = \frac{b(0)\Delta t}{2\sigma  x^\alpha }$ satisfying $y \leq y_0$, if $x>
\frac{b(0)}{\sqrt {2}}\Delta t$.  Then
\begin{align*}
&\Pp\left(\inf\limits_{0< s\leq
\Delta  t}\frac{x^{1-\alpha }}{\sigma }+\frac{(b(0)-Kx)}{\sigma  x^\alpha }s+
W_s\leq \frac{b(0)\Delta t}{2\sigma  x^\alpha }\right) \\
&= \frac{1}{2}\text{erfc}\left(\frac{x-\frac{b(0)\Delta t}{2}}{\sigma  x^\alpha \sqrt {2\Delta 
t}}+\frac{\left(b(0)-Kx\right)\sqrt {\Delta  t}}{\sigma  x^\alpha \sqrt {2}}\right)\\
&~~~+ \frac{1}{2}\exp\left(-\frac{2(b(0)-Kx)}{\sigma  ^2x^{2\alpha }}\left(x-\frac{b(0)\Delta  t}{2}\right)\right)\text{erfc}\left(\frac{x-\frac{b(0)\Delta t}{2}}{\sigma  x^\alpha \sqrt {2\Delta 
t}}-\frac{\left(b(0)-Kx\right)\sqrt {\Delta  t}}{\sigma  x^\alpha \sqrt {2}}\right)\\
&~~~:= A(x) + B(x) . 
\end{align*} 
For any  $z\geq 0$,
$\text{erfc}(z)\leq \exp(-z^2)$. Then,  if $\Dt$ is sufficiently small,
we have
\begin{align*}
A(x)
\leq \exp\left(-\frac{(x(1-K\Delta  t)+\frac{b(0)}{2}\Delta 
t)^2}{2\sigma  ^2x^{2\alpha}\Delta  t}\right) \leq 
\exp\left(-\frac{x^{2(1-\alpha )}}{8\sigma  ^2\Delta  t}\right)
\end{align*}
and, for $x\geq \frac{b(0)}{\sqrt {2}}\Delta t$, $A(x)
\leq\exp\left(-\frac{2^\alpha  b(0)^{2(1-\alpha )}}{16\sigma  ^2~\Delta 
t^{2\alpha -1}}\right)=\Oe(\Dt)$. 
Now we consider $B(x)$.  If  $x\geq \frac{\frac{3}{2}b(0)\Delta t}{(1+K\Delta  t)}$, then 
as for $A(x)$, we have 
\begin{align*}
B(x) &\leq  \exp\left(-\frac{2(b(0)-Kx)}{\sigma  ^2x^{2\alpha }}\left(x-\frac{b(0)\Delta 
t}{2}\right)\right)
\exp\left(-\frac{(x-\frac{b(0)\Delta t}{2}-(b(0)-Kx)
\Delta t)^2}{\sigma ^2 x^{2\alpha }2\Delta  t}\right)\\
~&~~~= \exp\left(-\frac{(x(1-K\Delta t)+\frac{b(0)\Delta  t}{2})^2}{\sigma ^2 x^{2\alpha }2\Delta 
t}\right)\leq \exp\left(-\frac{x^{2(1-\alpha )}}{8\sigma  ^2\Delta  t}\right)
\end{align*}
and $B(x)\leq \exp\left(-\frac{2^\alpha ~b(0)^{2(1-\alpha )}}{16\sigma 
^2~\Delta t^{2\alpha -1}}\right)=\Oe(\Dt)$, for $x\geq \frac{\frac{3}{2}b(0)\Delta t}{(1+K\Delta  t)}$. \\
If $\frac{b(0)}{\sqrt {2}}\Delta t \leq x<\frac{\frac{3}{2}b(0)\Delta t}{(1+K\Delta  t)} $,
then $\frac{2(b(0)-Kx)}{\sigma  ^2x^{2\alpha }}(x-\frac{b(0)\Delta 
t}{2})\geq \frac{b(0)^2\Delta  t(\frac{1}{\sqrt {2}}-\frac{1}{2})}{\sigma 
^2x^{2\alpha }}$ and 
\begin{align*}
B(x) \leq  \exp\left(-\frac{2(b(0)-Kx)}{\sigma  ^2x^{2\alpha }}\left(x-\frac{b(0)\Delta 
t}{2}\right)\right)\leq \exp\left(-\frac{b(0)^2\Delta 
t(\frac{1}{\sqrt {2}}-\frac{1}{2})}{\sigma  ^2x^{2\alpha }}\right).
\end{align*}
For  $x\geq \frac{b(0)}{\sqrt {2}}\Delta t$, we get $B(x)
\leq \exp\left(-\frac{2^{2\alpha }b(0)^{2(1-\alpha )}(\frac{1}{\sqrt {2}}-\frac{1}{2})
(1+K\Delta  t)^{2\alpha }}{3^{2\alpha }\sigma  ^2(\Delta t)^{2\alpha -1}}\right)=\Oe(\Dt)$.
\end{proof}

\begin{lemma}\label{un-sur-Z}
Assume {\rm (H1), (H2)} and {\rm
(H3')}. Let $\tau $ be the stopping time defined in \eqref{def_tD}. For
all $p\geq 0$, there  
exists a positive constant $C$ depending on 
$b(0)$, $\sigma $, $\alpha $, $T$ and $p$ but not on $\Dt$, such that 
\begin{align}\label{un-sur-Zbis}
\forall  t\in[0,T], ~~\Ee\left(\frac{1}{\oZ_{t\wedge \tau }^p}\right)\leq C \left(1+\frac{1}{x_0^p}\right).
\end{align}
\end{lemma} 
\begin{proof}
First, we prove that 
\begin{align}\label{moitiealpha}
\forall  t\in [0,T],~~~
\Pp\left(\oZ_t\leq \frac{\oX_{\eta (t)}}{2}\right)\leq \Oe(\Delta  t).
\end{align}
Indeed, while proceeding as in  the proof of
Lemma~\ref{stoppingtime-1},  we have 
\begin{align*}
\Pp\left(\oZ_t\leq \frac{\oX_{\eta (t)}}{2} \right)
\leq
\Ee\exp\left(-\frac{\left(\oX_{\eta (t)}(1-2K(t-\eta (t)))+2b(0)(t-\eta (t))\right)^2}{8\sigma  
^2(t-\eta (t))\oX_{\eta (t)}^{2\alpha }}\right).
\end{align*}
By using $(a+b)^2\geq a ^2+2ab$, with $a=\oX_{\eta (t)}(1-2K(t-\eta (t)))$ and 
$b=2b(0)(t-\eta (t))$, 
\begin{align*}
\Pp\left(\oZ_t\leq \frac{\oX_{\eta (t)}}{2}\right)
\leq
\Ee\left(\exp\left(-\frac{\oX_{\eta (t)}^{2(1-\alpha )}(1-2K\Dt)^2}{8\sigma 
^2\Dt}\right)\exp\left(-\frac{b(0)(1-2K\Dt)}{2\sigma 
^2\oX_{\eta (t)}^{2\alpha -1}}\right)\right). 
\end{align*}
For $\Dt$ sufficiently small, 
\begin{align*}
\Pp\left(\oZ_t\leq \frac{\oX_{\eta (t)}}{2} \right)
\leq
\Ee\left(\exp\left(-\frac{\oX_{\eta (t)}^{2(1-\alpha )}}{32\sigma 
^2\Dt}\right)\exp\left(-\frac{b(0)}{4\sigma  ^2\oX_{\eta (t)}^{2\alpha -1}}\right)\right).
\end{align*}
By separating the events $\left\{\oX_{\eta (t)}\geq \sqrt {\Dt}\right\}$ and
$\left\{\oX_{\eta (t)}<\sqrt {\Dt}\right\}$  in the expectation above, we
obtain
\begin{align*}
\Pp\left(\oZ_t\leq \frac{\oX_{\eta (t)}}{2}\right)
\leq \exp\left(-\frac{1}{32\sigma ^2(\Delta t)^\alpha }\right)
+\exp\left(-\frac{b(0)}{4\sigma ^2(\Dt)^{\alpha -\frac{1}{2}}}\right) = \Oe(\Dt).
\end{align*}
Now we prove~\eqref{un-sur-Zbis}. Notice that $\oZ_{t\wedge  \tau } = \oX_{t\wedge 
\tau }$, by the It\^o's formula 
\begin{align*} 
\frac{1}{\oZ_{t\wedge  \tau }^p}= \frac{1}{x_0^p}-p\int _{0}^{t\wedge  \tau  }\frac{b(\oX_{\eta (s)})}{\oZ_s^{p+1}}ds -p\sigma  \int _{0}^{t\wedge  \tau  }\frac{\oX_{\eta (s)}^\alpha  }{\oZ_s^{p+1}}dW_s+p(p+1)\frac{\sigma  ^2}{2}\int _{0}^{t\wedge  \tau  }\frac{\oX_{\eta (s)}^{2\alpha } }{\oZ_s^{p+2}}ds.
\end{align*}
Taking the expectation and using again $b(x) \geq b(0) - Kx$, we have
\begin{align*}
\Ee\left(\frac{1}{\oZ_{t\wedge \tau }^p}\right)
\leq &\frac{1}{x_0^p}
- p\Ee\left(\int _{0}^{t\wedge  \tau }\frac{b(0)}{\oZ_s^{p+1}}ds\right)
+pK\Ee\left(\int _{0}^{t\wedge  \tau }\frac{\oX_{\eta (s)}}{\oZ_s^{p+1}}ds\right)\\
&+p(p+1)\frac{\sigma ^2}{2}\Ee\left(\int _{0}^{t\wedge  \tau  }\frac{\oX_{\eta (s)}^{2\alpha }
}{\oZ_s^{p+2}}ds\right).
\end{align*} 
By the definition of $\tau $ in~\eqref{def_tD}, 
\begin{align*}
\Ee\left(\int _{0}^{t\wedge  \tau }\frac{\oX_{\eta (s)}}{\oZ_s^{p+1}}ds\right) = \Ee\left(\int _{0}^{t\wedge \tau}\ind_{(\oZ_{s}\leq
\frac{\oX_{\eta (s)}}{2})}\frac{\oX_{\eta (s)}}{\oZ_{s}^{p+1}}ds\right)
+\Ee\left(\int _{0}^{t\wedge \tau}\ind_{(\oZ_{s}>
\frac{\oX_{\eta (s)}}{2})}\frac{\oX_{\eta (s )}}{\oZ_{s}^{p+1}}ds \right)\\
\leq \left(\frac{2}{b(0)\Dt}\right)^{p+1}T
\sup_{t\in [0,\,T]}\left[\Pp\left(\oZ_t\leq \frac{\oX_{\eta (t)}}{2}\right)\right]^{1/2}\sup_{t\in [0,\,T]}\left[\Ee\left(\oX_{\eta (t)}^2\right)\right]^{1/2}+2\int _{0}^{t}\Ee\left(\frac{1}{\oZ_{s\wedge \tau }^{p}}\right)ds.
\end{align*}
We conclude, by the Lemma~\ref{majo-moments-oX} and the upper-bound
\eqref{moitiealpha} that
\begin{align*}
\Ee\left(\int _{0}^{t\wedge  \tau }\frac{\oX_{\eta (s)}}{\oZ_s^{p+1}}ds\right) \leq C +
2\int _{0}^{t}\Ee\left(\frac{1}{\oZ_{s\wedge \tau }^{p}}\right)ds.
\end{align*}	
Similarly,
\begin{align*} 
\Ee\left(\int _{0}^{t\wedge  \tau  }\frac{\oX_{\eta (s)}^{2\alpha }}{\oZ_s^{p+2}}ds\right)
& =
\Ee \left(\int _{0}^{t\wedge  \tau  }\ind_{(\oZ_{s}\leq
\frac{\oX_{\eta (s)}}{2})}\frac{\oX_{\eta (s)}^{2\alpha }
}{\oZ_{s}^{p+2}}ds\right)
+\Ee \left(\int _{0}^{t\wedge  \tau  }\ind_{(\oZ_{s}>
\frac{\oX_{\eta (s)}}{2})}\frac{\oX_{\eta (s)}^{2\alpha }
}{\oZ_{s}^{p+2}}ds\right)\\
&\leq
\Ee \left(\int _{0}^{t\wedge  \tau  }\ind_{(\oZ_{s}\leq
\frac{\oX_{\eta (s)}}{2})}\frac{\oX_{\eta (s)}^{2\alpha }}{\oZ_{s}^{p+2}}ds\right)
+ 2^{2\alpha }\Ee\left(\int _{0}^{t\wedge  \tau  }\frac{ds}{\oZ_{s}^{p+2(1-\alpha )}}\right).
\end{align*}
By using again the Lemma~\ref{majo-moments-oX} and the upper-bound
\eqref{moitiealpha}, we have 
\begin{align*}
\Ee\left(\int _{0}^{t\wedge  \tau  }\ind_{(\oZ_{s}\leq
\frac{\oX_{\eta (s)}}{2})}\frac{\oX_{\eta (s)}^{2\alpha }}{\oZ_{s}^{p+2}}ds \right)
&\leq T\left(\frac{2}{b(0)\Dt}\right)^{p+2}\sup_{t\in
[0,\,T]}\left[\sqrt {\Pp\left(\oZ_t\leq
\frac{\oX_{\eta (t)}}{2}\right)}\sqrt {\Ee\left(\oX_{\eta (t)}^{4\alpha }\right)}\right]\\ 
&\leq T\left(\frac{2}{b(0)\Dt}\right)^{p+2}\Oe(\Dt)
\leq C. 
\end{align*}
Finally,
\begin{align*}
\Ee\left(\frac{1}{\oZ_{t\wedge \tau }^p}\right)\leq \frac{1}{x_0^p}
+ \Ee \int _0^{t\wedge  \tau } \left( -\frac{pb(0)}{\oZ_{s}^{p+1}}
+\frac{2^{2\alpha -1}p(p+1)\sigma  ^2}{\oZ_{s}^{p+2(1-\alpha )}}  \right) ds
+C \int _0^t\Ee\left(\frac{1}{\oZ_{s\wedge \tau }^{p}} \right) ds
+ C. 
\end{align*}
We can easily check that there exists a positive constant $C$
such that, for all $z >0$,
$\frac{-pb(0)}{z^{p+1}}+\frac{p(p+1)2^{2\alpha -1}\sigma ^2}{z^{p+2(1-\alpha )}}\leq C$.
Hence 
\begin{align*}
\Ee\left(\frac{1}{\oZ_{t\wedge 
\tau }^p}\right)\leq
\frac{1}{x_0^p}+ C \int _{0}^{t}\Ee\left(\frac{1}{\oZ_{s\wedge \tau }^{p}}\right)ds+C
\end{align*}
and we conclude by applying the Gronwall Lemma. 
\end{proof}

\subsection{Proof of Theorem~\ref{theorem-faible-HUW-gene}}
\label{end-proof}
As in the proof of Theorem~\ref{theorem-faible-CIR-gene}, we use the
Feynman--Kac  representation of the solution of the Cauchy problem 
\eqref{edp_u_alpha>1/2}, studied in the Proposition
\ref{proposition_edp_alpha>1/2}: for all 
$(t,x) \in [0,T]\times (0,+\infty )$, $\Ee f(X_{T-t}^x) = u(t,x)$. Thus, the
weak error becomes
\begin{align*}
\Ee f(X_T) - \Ee f(\oX_T) = \Ee \left(u(0,x_0) - u(T,\oX_T)\right). 
\end{align*}
Let $\tau$ be the stopping time defined in~\eqref{def_tD}. By
 Lemma~\ref{stoppingtime-1}, 
\begin{align*} 
\Ee \left( u(T\wedge \tD,\oZ_{T\wedge  \tD})- u(T,\oX_T)\right)
\leq  2 \left\|u\right\|_{L^\infty ([0,T]\times  [0,+\infty ])} \Pp\left(\tau  \leq T\right) \leq
\Oe(\Delta  t). 
\end{align*}	
We bound the error by 
\begin{align*} 
\left|\Ee \left(u(T,\oX_T)-u(0,x_0)\right)\right|
\leq 
|\Ee \left(u(T\wedge \tD,\oZ_{T\wedge  \tD})-u(0,x_0)\right) | + \Oe(\Delta  t)
\end{align*} 
and we are now interested in $\Ee \left(u(T\wedge \tD,\oZ_{T\wedge  \tD})-(u(0,x_0)\right)$. 
Let $L$  and $\LL_z$ the second order differential operators
defined for any $C^2$ function $g(x)$ by
\begin{align*} 
Lg(x) =b(x)\frac{\partial  g}{\partial  x}(x)+\frac{\sigma  ^2}{2}x^{2\alpha }\frac{\partial  ^2g}{\partial 
x^2}(x) ~~\mbox{ and }~~ 
\LL_z g(x)= b(z)\frac{\partial  g}{\partial  x}(x)+\frac{\sigma  ^2}{2}z^{2\alpha }\frac{\partial  ^2g}{\partial  x^2}(x).
\end{align*}
From Proposition~\ref{proposition_edp_alpha>1/2}, $u$ is in
$C^{1,4}([0,T]\times  (0,+\infty ))$ satisfying $\frac{\partial u}{\partial s}(t,x) + Lu(t,x) =0$.
$\oX_t$ has bounded moments and the stopped process
$(\oX_{t\wedge \tD})$=$(\oZ_{t\wedge  \tD})$ has negative moments.
Hence, applying the It\^o's formula, 
\begin{align*}
\Ee\left[u(T\wedge \tau,\,\oX_{T\wedge \tau })-u(0,x_0)\right]
=\Ee\int _0^{T\wedge \tau }\left(\LL_{\oX_{\eta (s)}} u-Lu\right)(s ,\oX_s)ds, 
\end{align*}
Notice that 
\begin{align*}
& \partial _\theta   (\LL_zu-Lu)+b(z)\partial _x (\LL_zu-Lu)
+\frac{\sigma ^2}{2}z^{2\alpha }\partial ^2_{x^2}(\LL_zu-Lu)\\
& =\LL_z^2u-2\LL_zLu+L^2u
\end{align*}
and by applying again  the It\^o's formula
between $\eta (s)$ and $s$ to $\left(\LL_{\oX_{\eta (s)}}
u-Lu\right)(s,\oX_{s})$, 
\begin{align*} 
&\Ee\left[u(T\wedge \tau,\,\oX_{T\wedge \tau })-u(0,\,x_0)\right]\\
&=
\int _0^T\int _{\eta (s)}^s\Ee\left[\ind_{\left(\theta \leq\tau \right)}\left(\LL_{\oX_{\eta (s)}}^2u-2\LL_{\oX_{\eta (s)}}
Lu+L^2u\right)(\theta ,\oX_{\theta })\right]d\theta  ds.
\end{align*}
$\left(\LL_z^2u-2\LL_zLu+L^2u\right)(\theta ,x)$ combines the 
derivatives of $u$
up to the order four with  $b$ and its derivatives up to the order two and
some power functions like the $z^{4\alpha }$ or $x^{2\alpha  -2}$.  When we value this
expression at the point $(z,x) = (\oX_{\eta (s\wedge  \tau ) },\oX_{\theta \wedge  \tau })$, with the
upper bounds on  
the derivatives of $u$ given in the
Proposition~\ref{proposition_edp_alpha>1/2} and the positive and negative 
moments of $\oX$ given in the Lemmas~\ref{majo-moments-oX} and \ref{un-sur-Z}, we get 
\begin{align*}
\left|\Ee\left[\ind_{\left(\theta \leq\tau \right)}\left(\LL_{\oX_{\eta (s)}}^2u-2\LL_{\oX_{\eta (s)}}Lu+L^2u\right)(\theta  ,\oX_{\theta  })\right]\right|\leq
C\left(1+\frac{1}{x_0^{q(\alpha )}}\right) 
\end{align*}
which implies the result of Theorem~\ref{theorem-faible-HUW-gene}. 
\appendix
\section{On the Cox-Ingersoll-Ross model}

In \cite{cox-ingersoll-al-85}, Cox, Ingersoll and  Ross proposed to
model the dynamics of the short term interest rate  as the solution of the
following stochastic differential equation
\begin{align}\label{cir}
\left\{ \begin{array}{l}
dr^x_t= (a -b r^x_t) dt + \sigma  \sqrt {r^x_t}dW_t, \\
r_0^x = x \geq 0, 
\end{array} \right. 
\end{align} 
where $(W_t,0\leq t\leq T)$ is a one-dimensional Brownian motion on
a probability space $(\Omega , \FF, \Pp)$, $a$ and $\sigma  $ are positive 
constants and $b\in \er$. For any $t\in [0,T]$, let
$\FF_t = \sigma  (s\leq t, W_s)$.
\begin{lemma}\label{moment-inverse-cir}
For any $x>0$ and any $p>0$, 
\begin{align}\label{egalite-moment-inverse-cir}
\begin{array}{ll}
\Ee\left[\frac{1}{(r^x_t)^p} \right]
 =&\frac{1}{\Gamma (p)}
\left( \frac{2 b}{\sigma  ^2 (1 - e^{- bt})}\right)^p \\
& ~~~\times \int _0^{1} (2-\theta ) \theta ^{p-1}(1-\theta  )^{\frac{2a}{\sigma ^2}-p-1}
\exp\left(-\frac{2 b x \theta  }{\sigma  ^2 (e^{bt} -1)}\right) d\theta ,
\end{array}	
\end{align}
where $\Gamma (p) = \int _0^{+\infty } u^{p-1} \exp(-u) du$, $p>0$, denotes the Gamma
function. 
Moreover, if $a > \sigma ^2$
\begin{align}\label{major-moment-inverse-cir-1}
\Ee\left[\frac{1}{r^x_t} \right] \leq  \frac{e^{bt}}{x}
\end{align} 
and, for any $p$ such that $1<p<\frac{2a}{\sigma ^2}-1$,  
\begin{align}\label{major-moment-inverse-cir-p}
\Ee\left[\frac{1}{(r^x_t)^p} \right] \leq \frac{1}{\Gamma (p)} \left( \frac{
2e ^{|b|t}}{\sigma ^{2}t}\right)^p\mbox{ or }
\Ee\left[\frac{1}{(r^x_t)^p} \right] \leq C(p,T) \frac{1}{x^p},
\end{align}
where $C(p,T)$ is a positive constant depending on $p$ and $T$.
\end{lemma} 
\begin{proof}
By the definition of the Gamma function, for all $x>0$ and
$p>0$, $x^{-p}= \Gamma (p)^{-1} \int _0^{+\infty } u ^{p-1}\exp(-ux)du$, so that
\begin{align*}
\Ee\left[\frac{1}{(r^x_t)^p} \right] = \frac{1}{\Gamma  (p)} \int _0^{+\infty  } u
^{p-1} \Ee\exp(-u r^x_t) du.
\end{align*}
The Laplace transform of $r^x_t$ is given by 
\begin{align*}
\Ee \exp(-u r^x_t)  =
\frac{1}{(2u L(t) +1)^{2a/\sigma ^2}}
\exp\left( -\frac{u L(t) \zeta  (t,x) }{2 u L(t) +1}\right), 
\end{align*}
where $L(t) = \frac{\sigma  ^2}{4b}(1 - \exp(-bt))$ and
$\zeta  (t,x) =\frac{4xb}{\sigma  ^2(\exp(bt) -1)}=xe ^{-bt}/L(t)$, (see
e.g. \cite{lamberton-lapeyre-96}). Hence, 
\begin{align*}
\Ee\left[\frac{1}{(r^x_t)^p} \right] =  
\frac{1}{\Gamma (p)}\int _0^{+\infty } \frac{u^{p-1}}{(2uL(t)+1)^{2a/\sigma ^2}}
\exp\left( -\frac{u L(t) \zeta  (t,x) }{2 u L(t) +1}\right) du. 
\end{align*}
By changing the variable
$\theta=2\frac{uL(t)}{2uL(t)+1}$ in
the integral above, we obtain 
\begin{align*}
\Ee\left[\frac{1}{(r^x_t)^p} \right]
 = \frac{1}{2^p\Gamma (p)L(t)^p}\int _0^{1} (2-\theta )\theta ^{p-1}(1-\theta  )^{\frac{2a}{\sigma ^2}-p-1}
\exp\left(-\frac{x e^{-bt}\theta  }{2 L(t)}\right) d\theta , 
\end{align*}
from which  we deduce \eqref{egalite-moment-inverse-cir}. Now if $a > 
\sigma  ^2$, we have for $p=1$
\begin{align*} 
\Ee\left[\frac{1}{r^x_t} \right] \leq  \frac{1}{2 L(t)}\int _0^{1}
\exp\left(-\frac{x e^{-bt}\theta  }{2 L(t)}\right) d\theta  \leq \frac{e^{bt}}{x}
\end{align*} 
and for $1<p<\frac{2a}{\sigma ^2}-1$, 
$\Ee\left[\frac{1}{(r^x_t)^p} \right]  \leq  \frac{1}{2^p\Gamma (p)L(t)^p} =
\frac{2^p |b|^p}{\sigma ^{2p}\Gamma (p)(1-e^{-|b|t})^p}$ 
which gives \eqref{major-moment-inverse-cir-p}, by noting that
$(1-e^{-|b|t})\geq |b|te ^{-|b|t}$.  
\end{proof}
\begin{lemma}\label{novikov-cir}
If $a \geq \sigma ^2/2$ and $b\geq 0$, there exists a constant $C$ depending on $a$,
$b$, $\sigma $ and  $T$, such that 
\begin{align}
\sup_{t\in [0,T]}\Ee\exp\left(\frac{\nu ^2 \sigma  ^2}{8}\int _0^t\frac{ds}{r^x_s}\right)
\leq  C\left( 1 + {x^{-\frac{\nu }{2}}} \right), 
\end{align}
where $\nu  = \frac{2a}{\sigma ^2}-1 \geq 0$. 
\end{lemma}
\begin{proof}
For any $t\in [0,T]$, we set $H_t=\frac{2}{\sigma  }\sqrt {r^x_t}$, so that
$\Ee\exp(\frac{\nu ^2 \sigma  ^2}{8}\int _0^t\frac{ds}{r^x_s})$
$= \Ee\exp(\frac{\nu ^2}{2}\int _0^t\frac{ds}{H^2_s})$. The process 
$(H_t,t\in [0,T])$ solves 
\begin{align*}
dH_t = \left(\frac{2a}{\sigma  ^2}-\frac{1}{2}\right)\frac{dt}{H_t}
-\frac{b}{2}H_tdt+dW_t,\;\;H_0=\frac{2}{\sigma }\sqrt {x}. 
\end{align*}
For any $t\in [0,T]$, we set $B_t = H_t - H_0 - \int _0^t
\left(\frac{2a}{\sigma  ^2}-\frac{1}{2}\right)\frac{ds}{H_s}$.
Let $(\ZZ_t,t\in [0,T])$ defined by
\begin{align*}
\ZZ_t = \exp\left(
- \int _0^t \frac{b}{2} H_s dB_s - \frac{b^2}{8}\int _0^t H^2_s ds \right).
\end{align*}
By the Girsanov Theorem, under the probability $\Qq$ such that
$\frac{ d \Qq}{ d \Pp}\bigg|_{\FF_t} = \frac{1}{\ZZ_t}$,
$(B_t,t\in [0,T])$ is a Brownian motion. Indeed $(H_t,t\in [0,T])$ solves
\begin{align*}
dH_t = \left(\frac{2a}{\sigma  ^2}-\frac{1}{2}\right)\frac{dt}{H_t}
+dB_t,\;t\leq T,\;\;H_0=\frac{2}{\sigma }\sqrt {x}
\end{align*}
and  under $\Qq$, we note that $(H_t)$ is a Bessel process with index
$\nu =\frac{2a}{\sigma  ^2}-1$. Moreover, by the integration by parts formula,
$\int _0^t 2 H_s dB_s = H^2_t - H^2_0 - \frac{4a}{\sigma  ^2}t$ and 
\begin{align*}
\ZZ_t = \exp\left(-\frac{b}{4}H^2_t  - \frac{b^2}{8}\int _0^t H^2_s ds
 + \frac{b}{\sigma  ^2}x + t\frac{b a}{\sigma  ^2}\right)
\leq \exp\left(\frac{b}{\sigma  ^2}x + T \frac{b a}{\sigma  ^2}\right). 
\end{align*}
Now, denoting by $\Ee^{\Qq}$ the expectation relative to $\Qq$,
\begin{align*}
\Ee\exp\left(\frac{\nu ^2}{2}\int _0^t\frac{ds}{H^2_s}\right) & =
\Ee ^{\Qq}
\left[\exp\left(\frac{\nu ^2}{2}\int _0^t\frac{ds}{H^2_s}\right)\ZZ_t\right] \\
& \leq \exp\left(\frac{b}{\sigma  ^2}x + T \frac{b a}{\sigma  ^2}\right)
\Ee^{\Qq}
\left[\exp\left(\frac{\nu ^2}{2}\int _0^t\frac{ds}{H^2_s}\right) \right].
\end{align*}
Let $\Ee_{\frac{2}{\sigma  }\sqrt {x}}^{(\nu )}$ denotes the expectation relative to
$\Pp_{\frac{2}{\sigma  }\sqrt {x}}^{(\nu )}$, the law on $C(\er^+,\er^+)$ of the
Bessel process with index $\nu $, starting at $\frac{2}{\sigma  }\sqrt {x}$.
The next step uses the following change of probability
measure, for $\nu  \geq 0$ (see Proposition 2.4 in \cite{geman-yor-93}).
\begin{align*}
\Pp_{\frac{2}{\sigma  }\sqrt {x}}^{(\nu )} \bigg|_{\sigma  (R_s, s\leq t)} =
\left(\frac{\sigma  R_t}{2 \sqrt {x}}\right)^\nu  \exp\left( - \frac{\nu  ^2}{2}\int _0^t
\frac{ds}{R_s^2}\right) \Pp_{\frac{2}{\sigma  }\sqrt {x}}^{(0)}\bigg|_{\sigma  (R_s,
s\leq t)}, 
\end{align*}
where $(R_t,t\geq 0)$ denotes the canonical process on
$C(\er^+,\er^+)$. Then, we obtain that 
\begin{align*}
\Ee\exp\left(\frac{\nu ^2}{2}\int _0^t\frac{ds}{H^2_s}\right)
& \leq 
\exp\left(\frac{b}{\sigma  ^2}x + T \frac{b a}{\sigma  ^2}\right)
\Ee_{\frac{2}{\sigma  }\sqrt {x}}^{(\nu )}
\left[\exp\left(\frac{\nu ^2}{2}\int _0^t\frac{ds}{R^2_s}\right) \right]\\ 
& \leq 
\exp\left(\frac{b}{\sigma  ^2}x + T \frac{b a}{\sigma  ^2}\right)
\Ee_{\frac{2}{\sigma  }\sqrt {x}}^{(0)}\left[
\left(\frac{\sigma  R_t}{2 \sqrt {x}}\right)^\nu \right]. 
\end{align*}
It remains to compute $\Ee_{\frac{2}{\sigma  }\sqrt {x}}^{(0)}\left[
\left(\frac{\sigma  R_t}{2 \sqrt {x}}\right)^\nu \right]$.
Let $(W^1_t,W^2_t,t\geq 0)$ be a two 
dimensional Brownian motion. Then
\begin{align*}
\Ee_{\frac{2}{\sigma  }\sqrt {x}}^{(0)}\left[\left(\frac{\sigma  R_t}{2
\sqrt {x}}\right)^\nu \right] = \left(\frac{\sigma }{2 \sqrt {x}}\right)^\nu 
\Ee\left[ \left((W^1_t)^2
+ (W^2_t + \frac{2\sqrt {x}}{\sigma  })^2\right)^{\frac{\nu }{2}}\right]
\end{align*}
and an easy computation shows that 
$\Ee_{\frac{2}{\sigma  }\sqrt {x}}^{(0)}\left[\left(\frac{\sigma  R_t}{2
\sqrt {x}}\right)^\nu \right] \leq C(T) \left(1 + x^{-\frac{\nu }{2}}\right)$. 
\end{proof}
\section{Proofs of Propositions \ref{justification-cir} and \ref{justification-hull-white}}
\label{proofs-justification}
\begin{proof}[Proof of Proposition \ref{justification-cir}]
To simplify the presentation, we consider only the case when $k(x)$ and
$h(x)$ are nil. For any $\epsilon >0$ and $x>0$, we define  for all $t\in
[0,\,T]$, the process
$J_t^{x,\epsilon}=\frac{1}{\epsilon}(X_t^{x+\epsilon}-X_t^x)$, satisfying 
\begin{align*} 
J_t^{x,\epsilon}=1+\int _0^t\phi_s^\epsilon
J_s^{x,\epsilon}ds+\int _0^t\psi_s^\epsilon J_s^{x,\epsilon}dW_s, 
\end{align*} 
with $\phi_s^\epsilon =\int _0^1b'\left(X_s^x+\theta   \epsilon
J_t^{x,\epsilon}\right)d\theta $
and $
\psi _s^\epsilon = \int _0^1\frac{\sigma d\theta  }{2\sqrt
{X_s^x+\epsilon \theta  J_t^{x,\epsilon}}}$. Under {\rm (H3)}, the
trajectories  $(X_t^x,0\leq t\leq T)$  are strictly 
positive a.s. (see Remark \ref{H3}). 
By  Lemma \ref{moment-inverse-cir}, $\int _0^t\psi_s^\epsilon dW_s$ is a
martingale. Then $J_t^{x,\epsilon}$ is explicitly given by 
\begin{align*}
J_t^{x,\epsilon}=\exp\left(\int _0^t\phi_s^\epsilon
ds+\int _0^t\psi_s^\epsilon dW_s-\frac{1}{2}\int _0^t(\psi_s^\epsilon)^2ds\right).
\end{align*}
We remark that $\int _0^t \frac{\sigma  }{2 \sqrt  {X_s^x}} dW_s = \frac{1}{2} \log
\left(\frac{X_t^x}{x}\right) - \int _0^t \frac{1}{2} \frac{b(X_s^x)}{X_s^x} 
ds$ and  
\begin{align*} 
J_t^{x,\epsilon}\leq C\sqrt {\frac{X_t^x}{x}}\exp\left(- \int _0^t \frac{1}{2}\frac{b(X_s^x)}{X_s^x}ds
-\frac{1}{2}\int _0^t(\psi_s^\epsilon)^2ds
+\int _0^t\left(\psi_s^\epsilon - \frac{\sigma  }{2\sqrt {X_s^x}}\right) dW_s\right).
\end{align*} 	
We upper-bound the moments $\Ee(J_t^{x,\epsilon})^\alpha $, $\alpha  >0$.
As $b(x) \geq b(0) - Kx$, for any $p$, 
\begin{align*}
(J_t^{x,\epsilon})^\alpha  \leq &
C \left(\frac{X_t^x}{x}\right)^{\frac{\alpha }{2}}
\exp\left( - \int _0^t \frac{\alpha  }{2}\frac{b(0)}{X_s^x}ds
-\frac{\alpha  }{2}\int _0^t(\psi_s^\epsilon)^2ds + \int _0^t \frac{\alpha ^2p}{2} \left(\psi_s^\epsilon - \frac{\sigma 
}{2\sqrt {X_s^x}}\right)^2 ds\right)\\
& \times  \exp\left( \int _0^t \alpha  \left(\psi_s^\epsilon - \frac{\sigma 
}{2\sqrt {X_s^x}}\right) dW_s
- \int _0^t \frac{\alpha ^2p}{2} \left(\psi_s^\epsilon - \frac{\sigma 
}{2\sqrt {X_s^x}}\right)^2 ds \right) 
\end{align*}
and by the   H\"older  Inequality for  $p>1$, we have
\begin{align*}
\Ee (J_t^{x,\epsilon})^\alpha 
\leq & C \left\{\Ee \left[
\left(\frac{X_t^x}{x}\right)^{\frac{\alpha p}{2(p-1)}}\exp\left(\frac{\alpha  p }{2(p-1)}\left[ 
- \int _0^t \frac{b(0)}{X_s^x}ds
 + \alpha  p \int _0^t \frac{\sigma ^2}{4 X_s^x} ds\right]\right)\right]\right\}^{\frac{p-1}{p}}. 
\end{align*}
Then, for any $0<\alpha < 4$, for any $p>1$ such that $\alpha  p \leq 4$,
\begin{eqnarray}\label{Jalpha}
\Ee (J_t^{x,\epsilon})^\alpha  \leq & C \left\{ \Ee \left[
\left(\frac{X_t^x}{x}\right)^{\frac{\alpha p}{2(p-1)}}\right]\right\}^{\frac{p-1}{p}}. 
\end{eqnarray}
The same computation shows that for the same couple $(p,\alpha )$ and for any $0\leq \beta  \leq
\frac{p-1}{p}$,
\begin{eqnarray}\label{Jalphabeta}
\Ee \left(\frac{ (J_t^{x,\epsilon})^\alpha  }{(X^x_t)^{\frac{\alpha  }{2} + \beta }}\right) 
\leq  \frac{C}{x^\frac{\alpha }{2}}\left\{ \Ee \left[
(X^x_t)^{-\frac{\beta p}{p-1}}\right]\right\}^{\frac{p-1}{p}}, 
\end{eqnarray}
which is bounded according to Lemma \ref{moment-inverse-cir-gene}.
Hence, by \eqref{Jalpha} for $(\alpha ,p)=(2,2)$,  there exists a positive constant $C$ such that 
\begin{align*}
\Ee(X_t^{x+\epsilon}-X_t^x)^2\leq C\epsilon^2,~~\forall  ~t\in [0,T], 
\end{align*}
and $X_t^{x+\epsilon}$ tends $X_t^x$ in probability.
We consider now the process $(J^x_t,t\in[0,T])$ solution of
\eqref{derivee-flot-eds}. Applying the integration by parts formula  in
\eqref{derivee-flot}, we obtain  that  $J_t^x =\sqrt {\frac{X^x_t}{x}}
\exp(\int _0^t (b^\prime (X^x_s) -\frac{b(X^x_s)}{2 X^x_s}  
+ \frac{\sigma ^2}{8}\frac{1}{X_s^x})ds)$, from which by using {\rm (H3)}, we 
have 
\begin{align}\label{Jbound}
J_t^x \leq  \sqrt {\frac{X^x_t}{x}}
\exp\left(-\int _0^t(b(0)-\frac{\sigma ^2}{4})\frac{ds}{2X_\theta ^x}\right)\exp(KT)
\leq C \sqrt {\frac{X^x_t}{x}}. 
\end{align}
Moreover, 
\begin{align*}
J_t^x  - J_t^{x,\epsilon}  = & \int _0^t b'(X^x_s)(J_s^x  - J_s^{x,\epsilon}) ds
+  \int _0^t \frac{\sigma }{2\sqrt{X^x_s}}(J_s^x -J_s^{x,\epsilon}) dW_s \\
& + \int _0^t (b'(X^x_s)-\phi  ^\epsilon_s) J^{x,\epsilon}_s ds
+  \int _0^t (\frac{\sigma }{2\sqrt{X^x_s}}-\psi^\epsilon_s) J^{x,\epsilon}_s dW_s.
\end{align*}
We study the convergence of $\Ee(J_t^x  - J_t^{x,\epsilon})^2$ as
$\epsilon$ tends to $0$. We set $\Verr^x_t := |J_t^x  -
J_t^{x,\epsilon}|$.  By the It\^o's  formula, 
\begin{align*}
\Ee(\Verr_t^x)^2  = & \Ee \int _0^t 2 b'(X^x_s) (\Verr_s^x)^2 ds
+ \Ee\int _0^t 2 (b'(X^x_s)-\phi  ^\epsilon_s) J^{x,\epsilon}_s(J_s^x  -
J_s^{x,\epsilon}) ds \\
& + \Ee \int _0^t \left(\frac{\sigma }{2\sqrt{X^x_s}}(J_s^x -J_s^{x,\epsilon}) +
(\frac{\sigma }{2\sqrt{X^x_s}}-\psi^\epsilon_s) J^{x,\epsilon}_s \right)^2 ds. 
\end{align*}
We upper-bound the third term in the right-hand side of the expression
above: as $ \frac{\sigma }{2\sqrt{X^x_s}}\geq \psi^\epsilon_s$ and
$(\frac{\sigma }{2\sqrt{X^x_s}}+\psi^\epsilon_s)\leq C/\sqrt {X^x_s}$,
\begin{align*}
\Ee \left( (\frac{\sigma }{2\sqrt{X^x_s}}-\psi^\epsilon_s)
J^{x,\epsilon}_s\right)^2\leq  C\Ee \left(
(\frac{\sigma }{2\sqrt{X^x_s}}-\psi^\epsilon_s)
\frac{(J^{x,\epsilon}_s)^2}{\sqrt {X^x_s}}\right). 
\end{align*}	
An  easy computation shows that
$\sqrt {X_s^x}(\frac{\sigma }{2\sqrt{X^x_s}}-\psi^\epsilon_s) \leq
\sqrt {\epsilon}\frac{\sqrt {J_s^{x,\epsilon}}}{\sqrt {X_s^x}}$. 
Then, 
\begin{align*}
\Ee \left( (\frac{\sigma }{2\sqrt{X^x_s}}-\psi^\epsilon_s)
J^{x,\epsilon}_s\right)^2 \leq C\sqrt {\epsilon} \Ee
\left(
\frac{(J^{x,\epsilon}_s)^{\frac{5}{2}}}{(X^x_s)^{\frac{3}{2}}}\right)
= C\sqrt {\epsilon} \Ee
\left(
\frac{(J^{x,\epsilon}_s)^{\frac{5}{2}}}{(X^x_s)^{\frac{5}{4} +
\frac{1}{4}}}\right) \leq C \sqrt{\epsilon}, 
\end{align*}
where we have applied \eqref{Jalphabeta} with $(\alpha,p,\beta) =
(\frac{5}{2}, \frac{8}{5}, \frac{1}{4} \leq \frac{p-1}{p} =
\frac{3}{8})$. 
By using the same arguments with \eqref{Jbound}, 
\begin{align*}
& \Ee\left(\frac{\sigma }{2\sqrt{X^x_s}}(J_s^x -J_s^{x,\epsilon})
(\frac{\sigma }{2\sqrt{X^x_s}}-\psi^\epsilon_s) J^{x,\epsilon}_s \right) 
 \leq 
\Ee \left(\frac{\sigma }{2\sqrt{X^x_s}}(J_s^x +J_s^{x,\epsilon}) \sqrt {\epsilon}\frac{\sqrt {J_s^{x,\epsilon}}}{X_s^x}
 J^{x,\epsilon}_s \right) \\
& \leq C \sqrt {\epsilon} \left( \Ee\left(
\frac{(J_s^{x,\epsilon})^{\frac{3}{2}}}{X^x_s}\right) + \Ee\left(
\frac{(J_s^{x,\epsilon})^{\frac{5}{2}}}{(X^x_s)^{\frac{3}{2}}}\right)\right)\leq
C \sqrt{\epsilon}, 
\end{align*}
where we have applied \eqref{Jalphabeta} with $(\alpha,p,\beta) =
(\frac{3}{2}, \frac{8}{3}, \frac{1}{4} \leq \frac{p-1}{p}
=\frac{5}{3})$.  
An easy computation shows that 
$|b'(X^x_s) - \phi_s^\epsilon| \leq \epsilon J_s^{x,\epsilon}
\|b''\|_{\infty }$. Coming back to the upper-bound of $\Ee(\Verr_t^x)^2$, we have
\begin{align}\label{J_prelim}
\Ee(\Verr_t^x)^2  \leq  C \int _0^t \Ee(\Verr_s^x)^2 ds
+ C \sqrt {\epsilon} t 
+ \Ee\left(\int _0^t \frac{\sigma  ^2 }{4 X^x_s}(\Verr^x_s)^2 ds \right). 
\end{align}
To conclude on the convergence, as $\epsilon$ tends to 0, we use the  
stochastic time change technique introduced in \cite{berkaoui-bossy-al-07} to analyze the
strong rate of convergence. For any $\lambda  >0$, we define the stopping time $\tau _\lambda  $ as
\begin{eqnarray*}
\tau _\lambda  = \inf\{s\in[0,T],~\gamma  (s) \geq \lambda  \}\mbox{ with } \gamma (t) = \int _0^t
\frac{\sigma  ^2 ds}{4 X^x_s}\mbox{ and } \inf\emptyset  = T.
\end{eqnarray*}
Then, by using the Lemma \ref{moment-inverse-cir-gene} with the Markov Inequality, 
\begin{align*}
\Pp(\tau _\lambda  < T)  &= \Pp( \gamma (T) \geq \lambda  )  \leq \exp(-\frac{\lambda }{2} ) \Ee \left(
\exp\left(\int _0^T \frac{\sigma  ^2 ds}{8 X^x_s}\right)\right) \leq  C
\exp(-\frac{\lambda }{2}). 
\end{align*}
Choosing $\lambda  =
-\log (\epsilon ^r)$ for a given $r>0$, we have that $\Pp(\tau _\lambda  < T) \leq C
\epsilon^\frac{r}{2}$ and
\begin{eqnarray*}
\Ee(\Verr^x_T)^2\leq \Ee(\Verr_{\tau _\lambda }^x)^2 + C \epsilon^\frac{r}{4}. 
\end{eqnarray*}	
With \eqref{J_prelim}, we can easily check that for any bounded stopping time $\tau \leq T$,
\begin{align*}
\Ee(\Verr_\tau^x)^2 & \leq \int _0^T \exp( C(T-s)) \left\{ 
\Ee\left( \int _0^\tau  \frac{\sigma  ^2 ds}{4 X^x_s} (\Verr^x_s)^2\right) + C\sqrt {\epsilon}\right\}
\end{align*}
and for $\tau _\lambda $,  
\begin{align*} 
\Ee(\Verr_{\tau _\lambda  }^x)^2 &
\leq C_1 \Ee\left(\int _0^{\tau _\lambda }(\Verr^x_s)^2 d\gamma  (s) \right)  + C_0
\sqrt {\epsilon}, 
\end{align*}
for some positive constants $C_0$ and $C_1$,  depending on $T$. 
After the change of time $u = \gamma  (s)$,  we can apply the Gronwall Lemma
\begin{align*}
\Ee(\Verr_{\tau _\lambda  }^x)^2
\leq C_1 \Ee\left( \int _0^\lambda  (\Verr^x_{\tau _u})^2 du\right) + C_0
\sqrt {\epsilon} \leq T C_0 \sqrt {\epsilon} \exp(C_1 \lambda ). 
\end{align*}
With the choice $r = (4 C_1)^{-1}$ and  $\lambda  = -\log (\epsilon ^r)$, we
get $\Ee(\Verr_{\tau _\lambda  }^x )^2\leq T C_0 \epsilon^{\frac{1}{4}}$.
As $T$ is arbitrary in the preceding reasoning, we conclude that $\Ee|J_t^x-J_t^{x,\epsilon}|$ 
tends to $0$ with
$\epsilon$ for all $t\in[0,T]$.  Consider now 
\begin{align*}
&\frac{g(X_t^{x+\epsilon})-g(X_t^x)}{\epsilon} -
g'(X_t^x)J_t^x =J_t^{x,\epsilon}\int _0^1g'(X_t^x+\epsilon 
\alpha  J_t^{x,\epsilon})d\alpha  - J_t^xg'(X_t^x) \\
&=(J_t^{x,\epsilon}-J_t^x)\int _0^1g'(X_t^x+\epsilon 
\alpha  J_t^{x,\epsilon})d\alpha +J_t^x\int _0^1\left(g'(X_t^x+\epsilon 
\alpha  J_t^{x,\epsilon})-g'(X_t^x)\right)d\alpha \\
& :=A^\epsilon + B^\epsilon.
\end{align*}
$\Ee A^\epsilon \leq \|g'\|_\infty   \Ee|J_t^x-J_t^{x,\epsilon}|$,  which
tends to zero with $\epsilon$. $B^\epsilon$ is a uniformly integrable
sequence. $g'$ is a continuous function. By the Lebesgue
Theorem, as $X_t^{x+\epsilon}$ tends $X_t^x$ in probability,
$B^\epsilon$ tends to $0$ with $\epsilon$.
As a consequence,  $\Ee(\frac{g(X_t^{x+\epsilon})-g(X_t^x)}{\epsilon})$ 
tends to $\Ee[g'(X_t^x)J_t^x]$ when $\epsilon$ tends to  $0$. 
\end{proof}

\begin{proof}[Proof of Proposition \ref{justification-hull-white}]
The proof is very similar to the proof of Proposition
\ref{justification-cir}. Again, we consider only the case when $h(x)$
and $k(x)$ are nil. Let
$J_t^{x,\epsilon}=\frac{1}{\epsilon}(X_t^{x+\epsilon}-X_t^x)$, given
also by 
\begin{align*}
J_t^{x,\epsilon}=\exp\left(\int _0^t\phi _s^\epsilon ds+\int _0^t \psi _s^\epsilon dW_s-\frac{1}{2}\int _0^t(\psi _s^\epsilon)^2ds\right),
\end{align*}
with $\phi _s^\epsilon =\int _0^1b'(X_t^x+ \theta   \epsilon J_s^{x,\epsilon})d\theta $
and $\psi _t^\epsilon = 
\int _0^1\frac{\alpha  \sigma  d\theta   }{(X_t^x+\theta \epsilon J_s^{x,\epsilon})^{1-\alpha }}$. 
For any $C^1$ function $g(x)$ with bounded derivative, we have 
\begin{align*}
&\frac{g(X_t^{x+\epsilon})-g(X_t^x)}{\epsilon} -
g'(X_t^x)J_t^x =J_t^{x,\epsilon}\int _0^1g'(X_t^x+\epsilon 
\theta   J_t^{x,\epsilon})d\theta   - J_t^xg'(X_t^x) \\
&=(J_t^{x,\epsilon}-J_t^x)\int _0^1g'(X_t^x+\epsilon 
\theta  J_t^{x,\epsilon})d\theta +J_t^x\int _0^1\left(g'(X_t^x+\epsilon 
\theta  J_t^{x,\epsilon})-g'(X_t^x)\right)d\theta \\
& :=A^\epsilon + B^\epsilon.
\end{align*}
$\Ee(J_t^{x,\epsilon})^\alpha  \leq \exp(\alpha \|b'\|_\infty  t)\Ee
\exp(\alpha \int _0^t
\psi _s^\epsilon dW_s- \frac{\alpha}{2}\int _0^t(\psi _s^\epsilon)^2ds)$ and, by using Lemma
\ref{minoration_1}{\it{(ii)}}, one  easily concludes that
$\Ee(J_t^{x,\epsilon})^\alpha  \leq C$ and consequently that $X_t^{x+\epsilon}$ converges to
$X_t^x$ in $L^2(\Omega )$. Then,  by applying the Lebesgue Theorem,
$\Ee|B^\epsilon|$ tends  to $0$. 
Moreover, 
$\Ee|A^\epsilon| \leq \|g'\|_\infty  \sqrt {\Ee
|J_t^{x,\epsilon}-J_t^x|^2}$.
We can proceed as  in the proof of the
Proposition \ref{justification-cir}, to 
show that  $\Ee |J_t^{x,\epsilon}-J_t^x|^2 $ tends to 0, but now the
moments $\Ee(J_t^{x,\epsilon})^\alpha, \alpha >0 $ are bounded and the 
Lemma \ref{moment-inverse-HUW-gene} ensures that the $\Ee |
X_t^x|^{-p}, p>0$ are all bounded. 
\end{proof}

\section{End of the proof of Proposition \ref{proposition_edp_alpha=1/2}}\label{end_proof_deriv4}
To compute $\frac{\partial  ^4u}{\partial  x^4}(t,x)$, we need first to avoid the
appearance of  $J^x_{t}(1)$ in the expression of $\frac{\partial  ^3u}{\partial 
x^3}(t,x)$. We transform  
the expression of $\frac{\partial  ^3u}{\partial  x^3}(t,x)$ in
\eqref{expression-derivee-3-en-x}, in order to obtain $\frac{\partial  ^3u}{\partial 
x^3}(t,x)$ as a sum  of terms of the form
\begin{align*} 
&\Ee\left( \exp\left(\int _0^{T-t}\beta (X^x_s(1)) ds\right) \Gamma  (X^x_{T-t}(1))
J^x_{T-t}(1)\right) \\
& + \int _0^{T-t} \Ee\left\{ \exp\left(\int _0^s \beta 
(X^x_u(1)) du \right) J^x_s (1) \Lambda  (X^x_s(1)) \right\} ds
\end{align*} 
for some functions $\beta (x)$, $\Gamma (x)$, $\Lambda (x)$. In this first step, to
simplify the writing, we
write $X^x_s$ instead of $X^x_s(1)$.  Two
terms are not of this 
form in \eqref{expression-derivee-3-en-x}: 
\begin{align*}
\textrm{I}  &= 2\Ee\left\{\exp\left(2\int _0^{T-t}b'(X^x_s) ds\right)
f''(X^x_{T-t}) \int _0^{T-t} b''(X^x_s) J^x_sds\right\} \\
\textrm{II}  &= 
2\Ee\Big\{ \int _0^{T-t}
\exp\left(2\int _0^s b'(X^x_u)du\right)
 \frac{\partial  u}{\partial  x}(t+s,X^x_s)b''(X^x_s)\left(\int _0^s
b''(X^x_u)J^x_u du\right) ds  \Big\}. 
\end{align*}
The integration by parts formula gives immediately that
\begin{align*}
\textrm{II} = 2 \Ee\left\{  \int _0^{T-t} 
b''(X^x_s) J^x_s 
\left(\int _s^{T-t}  \frac{\partial  u}{\partial  x} (t+u,X^x_u) \exp\left(2\int _0^u
b'(X^x_\theta  ) d\theta  \right)b''(X^x_u) du \right)
ds \right\}.
\end{align*}
By using again the Markov property and the time homogeneity of the
process $(X^x_t)$, 
\begin{align*}
\left.\Ee\left[\exp\left(2\int _s^{T-t} b'(X^x_\theta  ) d\theta 
\right)f''(X^x_{T-t}) \right/ \FF_s\right]
=\Ee\left[ \exp\left(2\int _0^{T-t-s} b'(X^y_\theta  ) d\theta 
\right)f''(X^y_{T-t-s})\right]\bigg|_{y = X_s^x}
\end{align*}
and, by using \eqref{new-expression-derivee-2-en-x},
\begin{align*}
\textrm{I} = & 2 \int _0^{T-t} \Ee\left\{
b''(X^x_s) J^x_s \exp\left(2\int _0^s b'(X^x_\theta  ) d\theta  \right)
\frac{\partial  ^2 u}{\partial  x^2}(t+s,X^x_s) \right\} ds \\
& - 2 \int _0^{T-t} \Ee\left\{
b''(X^x_s) J^x_s \exp\left(2\int _0^s b'(X^x_\theta  ) d\theta 
\right)\right. \\
& \left. \hspace{1cm}
\times  \left(\int _0^{T-t-s} \Ee\left[ \frac{\partial  u}{\partial  x} (t+s+u,X^y_u)
\exp\left(2\int _0^u b'(X^y_\theta  ) d\theta  \right)b''(X^y_u) \right]
\bigg|_{y = X^x_s} du \right)\right\} ds.
\end{align*}
Conversely, 
\begin{align*}
& \int _0^{T-t-s}\Ee\left[ \frac{\partial  u}{\partial  x} (t+s+u,X^y_u)
\exp\left(2\int _0^u b'(X^y_\theta  ) d\theta  \right)b''(X^y_u) \right]
\bigg|_{y = X^x_s} du \\
& = \int _s^{T-t} \Ee\left. \left[ \frac{\partial  u}{\partial  x} (t+u,X^x_u) \exp\left(2\int _s^u
b'(X^x_\theta  ) d\theta  \right)b''(X^x_u)\right/\FF_s \right] du
\end{align*}
and then
\begin{align*}
\textrm{I} = & 2\int _0^{T-t} \Ee\left\{
b''(X^x_s) J^x_s \exp\left(2\int _0^s b'(X^x_\theta  ) d\theta  \right)
\frac{\partial  ^2 u}{\partial  x^2}(t+s,X^x_s) \right\} ds \\
& - 2\Ee\left\{  \int _0^{T-t} 
b''(X^x_s) J^x_s 
\left(\int _s^{T-t}  \frac{\partial  u}{\partial  x} (t+u,X^x_u) \exp\left(2\int _0^u
b'(X^x_\theta  ) d\theta  \right)b''(X^x_u) du \right)
ds \right\}.
\end{align*}
Finally, replacing $\textrm{I}$ and $\textrm{II}$ in
\eqref{expression-derivee-3-en-x}, we get
\begin{align*}
\frac{\partial  ^3u}{\partial  x^3}(t,x) = &
\Ee\left\{\exp\left(2\int _0^{T-t}b'(X^x_s) ds\right)
f^{(3)}(X^x_{T-t})J^x_{T-t}\right\} \\
& + \int _0^{T-t}\Ee\left\{
\exp\left(2\int _0^s b'(X^x_u)du\right)J^x_s
 \right. \\
& \hspace{1.5cm} \left. 
\times  \left( 3 \frac{\partial  ^2u}{\partial  x^2}(t+s,X^x_s) b''(X^x_s) + \frac{\partial  u}{\partial  x}(t+s,X^x_s) b^{(3)}(X^x_s)\right)
\right\}ds. 
\end{align*} 
To eliminate $J^x_{t}$, we introduce the probability $\Qq^{3/2}$ such that
$\frac{ d\Qq^{3/2}}{
d\Pp}\bigg|_{\FF_t}=\frac{1}{\ZZ_t^{(1,\frac{3}{2})}}$. Then 
\begin{align*}
\frac{\partial  ^3u}{\partial  x^3}(t,x) = &
\Ee^{3/2}\left\{\exp\left(2\int _0^{T-t}b'(X^x_s) ds\right)
f^{(3)}(X^x_{T-t})\ZZ_{T-t}^{(1,\frac{3}{2})}J^x_{T-t}\right\} \\
& + \int _0^{T-t}\Ee^{3/2}\left\{
\exp\left(2\int _0^s b'(X^x_u)du\right)\ZZ_s^{(1,\frac{3}{2})}J^x_s
 \right. \\
& \hspace{2.5cm}\left. 
\times   \left(3 \frac{\partial  ^2u}{\partial  x^2}(t+s,X^x_s) b''(X^x_s) +  \frac{\partial  u}{\partial  x}(t+s,X^x_s) b^{(3)}(X^x_s)\right)
\right\}ds. 
\end{align*}
Again, we note that  $\ZZ_t^{(1,\frac{3}{2})}J_{t}^x =
\exp\left(\int _0^{t} b'(X^x_u) du\right)$ 
and
\begin{align*}
\frac{\partial  ^3u}{\partial  x^3}(t,x) = &
\Ee^{3/2}\left\{\exp\left(3\int _0^{T-t}b'(X^x_s) ds\right)
f^{(3)}(X^x_{T-t})\right\} \\
& + \int _0^{T-t}\Ee^{3/2}\left\{
\exp\left(3\int _0^s b'(X^x_u)du\right)
 \right. \\
& \hspace{2.3cm}\left. 
\times  \left(3 \frac{\partial  ^2u}{\partial x^2}(t+s,X^x_s) b''(X^x_s)+
\frac{\partial  u}{\partial  x}(t+s,X^x_s) b^{(3)}(X^x_s)\right)\right\}ds
\end{align*}
where we write $X^x_\cdot$ instead of $X^x(1)_\cdot$. 
Finally, as $\LL^{\Qq^{3/2}}(X^x(1)) = \LL^{\Pp}(X^x(\frac{3}{2}))$, we
 obtain the following expression for $\frac{\partial  ^3u}{\partial  x^3}(t,x)$: 
\begin{align*} 
\frac{\partial  ^3u}{\partial  x^3}(t,x) = &
\Ee\left\{\exp\left(3\int _0^{T-t}b'(X^x_s(\frac{3}{2})) ds\right)
f^{(3)}(X^x_{T-t}(\frac{3}{2}))\right\} \\
& +  \int _0^{T-t}\Ee\left\{
\exp\left(3\int _0^s b'(X^x_u(\frac{3}{2})du\right)
\left(3 \frac{\partial  ^2u}{\partial  x^2}(t+s,X^x_s(\frac{3}{2}))
b''(X^x_s(\frac{3}{2})) \right. \right. \\
& \hspace{5.5cm}\left. \left. 
+ \frac{\partial  u}{\partial  x}(t+s,X^x_s(\frac{3}{2})) b^{(3)}(X^x_s(\frac{3}{2}))\right)
\right\}ds. 
\end{align*} 
$J^x_s(\frac{3}{2})$ exists and is given by \eqref{derivee-flot}.
By the Proposition \ref{justification-cir}, $\frac{\partial ^3u}{\partial  x^3}(t,x)$ is
continuously differentiable and  
\begin{align*}
 \frac{\partial  ^4u}{\partial  x^4}(t,x) = &
\Ee\left\{\exp\left(3\int _0^{T-t}b'(X^x_s(\frac{3}{2})) ds\right)
\right.\\
& \hspace{1cm}\left.
\times  \left[3f^{(3)}(X^x_{T-t}(\frac{3}{2}))\int _0^{T-t}
b''(X^x_s(\frac{3}{2}))J^x_s(\frac{3}{2}) ds  +
f^{(3)}(X^x_{T-t}(\frac{3}{2}))J^x_{T-t}\frac{3}{2})\right]\right\}\\
& +\int _0^{T-t} \Ee\left\{ \exp\left(3\int _0^s b'(X^x_u(\frac{3}{2}))du\right) 
 \left(3 \frac{\partial  ^2u}{\partial  x^2}(t+s,X^x_s(\frac{3}{2}))
b''(X^x_s(\frac{3}{2}))\right. \right. \\
& \hspace{3cm}\left. \left. 
+ \frac{\partial  u}{\partial  x}(t+s,X^x_s(\frac{3}{2})) b^{(3)}(X^x_s(\frac{3}{2}))\right)
\int _0^{s}3 
b''(X^x_u(\frac{3}{2}))J^x_u(\frac{3}{2}) du \right\}ds \\
& +\int _0^{T-t} \Ee\left\{ \exp\left(3\int _0^s
b'(X^x_u(\frac{3}{2}))du\right)
J^x_s(\frac{3}{2})
\left(3 \frac{\partial ^3u}{\partial  x^3}(t+s,X^x_s(\frac{3}{2}))
b''(X^x_s(\frac{3}{2})) \right. \right. \\
& \hspace{3cm}\left. \left. 
+ 4 \frac{\partial ^2u}{\partial  x^2}(t+s,X^x_s(\frac{3}{2}))
b^{(3)}(X^x_s(\frac{3}{2})) \right. \right. \\
& \hspace{3cm}\left. \left. 
+ \frac{\partial  u}{\partial  x}(t+s,X^x_s(\frac{3}{2}))
b^{(4)}(X^x_s(\frac{3}{2}))\right)
\right\}ds, 
\end{align*}
from which we can conclude on \eqref{borne_deriv_u_alpha=1/2}.

\bibliographystyle{plain}

\end{document}